\documentclass[a4paper, 11pt]{amsart} 
\usepackage[
  margin=1.9cm,
  includefoot,
  footskip=30pt,
]{geometry}
\usepackage{amsmath,amssymb,amscd}
\usepackage{amsfonts}
\usepackage[english]{babel}
\usepackage[leqno]{amsmath}
\usepackage{amssymb,amsthm}
\usepackage{mathrsfs} 
\usepackage{amscd}
\usepackage{enumerate}
\usepackage{epsfig}
\usepackage{relsize}
\usepackage{layout}
\usepackage[usenames,dvipsnames]{xcolor}
\usepackage[backref=page]{hyperref}
\usepackage{tikz}
\hypersetup{
 colorlinks,
 citecolor=Green,
 linkcolor=Red,
 urlcolor=Blue}
\usepackage[matrix,arrow,tips,curve]{xy}
\input{xy}
\xyoption{all}
\definecolor{darkgreen}{rgb}{0.0, 0.7, 0.0}
\definecolor{purple}{rgb}{0.5, 0.0, 0.5}
\definecolor{red}{rgb}{0.8, 0.2, 0.0}

\newtheorem{thm}{Theorem}[section]
\newtheorem{bthm}{Theorem}
\newtheorem{bcor}{Corollary}
\newtheorem{lemma}[thm]{Lemma}

\newtheorem{prop}[thm]{Proposition}

\newtheorem{claim}[thm]{Claim}

\numberwithin{equation}{section}
\theoremstyle{definition}
\newtheorem{defi}[thm]{Definition}

\newtheorem{notation}[thm]{Notation}

\theoremstyle{remark}
\newtheorem{remark}[thm]{Remark}

\newtheorem{example}[thm]{Example}
\newcommand{\Z}{\mathbb{Z}}
\newcommand{\C}{\mathbb{C}}
\newcommand{\Q}{\mathbb{Q}}
\newcommand{\R}{\mathbb{R}}

\newcommand{\Pic}{\operatorname{Pic}}

\DeclareMathOperator{\Hom}{{Hom}}
\DeclareMathOperator{\Ext}{{Ext}}
\DeclareMathOperator{\Spec}{Spec \:}

\def \Im{{\rm Im}}

\def \P{\mathbb{P}}

\def \F{\mathcal F}

\def\I{{\mathcal J}}
\def \L{\mathcal L}
\def \E{\mathcal E}
\def \G{\mathcal G}
\def \H{\mathcal H}

\def\O{\mathcal O}

\def\M0{\mathcal M^0}

\def\mapright#1{\smash{\mathop{\longrightarrow}\limits^{#1}}}

\newcommand{\SHom}{{\mathcal{H}om}}
\newcommand{\SExt}{{\mathcal{E}xt}}

\newcommand{\Id}{{\operatorname{Id}}}

\DeclareMathOperator{\depth}{{depth}}

\DeclareMathOperator{\codim}{{codim}}

\DeclareMathOperator{\Sing}{{Sing}}

\DeclareMathOperator{\Coker}{{Coker}}

\DeclareMathOperator{\Ker}{{Ker}}
\DeclareMathOperator{\Bs}{{Bs}}

\def\B{\mathbf{B}}

\newcommand{\rk}{\operatorname{rank}}

\begin{document}

\title[On the connectedness of some degeneracy loci and of Ulrich subvarieties]{On the connectedness of some degeneracy loci and of Ulrich subvarieties}

\author[V. Buttinelli, A.F. Lopez, R. Vacca]{Valerio Buttinelli*, Angelo Felice Lopez and Roberto Vacca**}

\address{\hskip -.43cm Valerio Buttinelli, Dipartimento di Matematica ``Guido Castelnuovo", Sapienza Universit\`a di Roma, Piazzale Aldo Moro 5, 00185 Roma, Italy. email: {\tt valerio.buttinelli@uniroma1.it}}

\address{\hskip -.43cm Angelo Felice Lopez, Dipartimento di Matematica e Fisica, Universit\`a di Roma
Tre, Largo San Leonardo Murialdo 1, 00146, Roma, Italy. e-mail {\tt angelo.lopez@uniroma3.it}}

\address{\hskip -.43cm Roberto Vacca, Dipartimento di Matematica, Universit\`a di Roma Tor Vergata, Via della Ricerca Scientifica, 00133 Roma, Italy. email: {\tt vacca@mat.uniroma2.it}}

\thanks{*The author thanks the ``Progetti di Avvio alla Ricerca 2024" of the University of Rome La Sapienza}

\thanks{The second and third authors were partially supported by the GNSAGA group of INdAM} 

\thanks{**Work supported by the MIUR Excellence Department Project MatMod@TOV awarded to the Department of Mathematics of the University of Rome Tor Vergata.}


\begin{abstract} 
We study connectedness of degeneracy loci $D_{r-k}(\varphi)$ of morphisms $\varphi : \O_X^{\oplus (r+1-k)} \to \E$, where $\E$ is a rank $r$ globally generated bundle on a smooth $n$-dimensional variety $X$ and $k \le 3$. For $k \le 2$ we give a characterization of connectedness in terms of vanishing of Chern classes. Moreover we prove that they are connected, for $k \le \min\{2, r-1,n-1\}$, if $\E$ is V-big. In the case of Ulrich bundles more precise results are given, both in general and in the case of surfaces.
\end{abstract}

\maketitle

\section{Introduction}

Let $\phi : E \to F$ be a morphism of vector bundles on a smooth complex variety $X$. A fundamental theorem of Fulton and Lazarsfeld \cite{fl} establishes, under the hypothesis that $\SHom(E,F)$ is ample, that the degeneracy loci of $\phi$ are nonempty if they have non-negative expected dimension and are connected if they have positive expected dimension. A few years later, Tu \cite{t} and Steffen \cite{ste} proved a similar result assuming $q$-ampleness of $\SHom(E,F)$. Several results have been proved since then, see for example \cite{d, fu, lay, ln}.

In the absence of some kind of ampleness property, things are more complicated. For example, when $\SHom(E,F)$ is just big and globally generated, degeneracy loci can be disconnected \cite{ln} also when they have positive expected dimension (even in the Ulrich case, see Example \ref{ese} inspired by \cite{ln}).  

In this paper we consider, for $1 \le k \le \min\{r,n\}$, injective morphisms 
$$\varphi : \O_X^{\oplus (r+1-k)} \to \E$$
where $\E$ is a rank $r$ globally generated bundle, and their degeneracy loci
$$D_{r-k}(\varphi)=\{x \in X : \rk \varphi(x) \le r-k\}.$$
which, as is well known, when $\varphi$ is general, are nonempty if and only if $c_k(\E) \ne 0$. 

We first give a small generalization, in this case, of \cite[Thm.~II(a)]{fl} in terms of stable base loci (that measure the positivity of $\E$): if $\B_+(\E) \ne X$, then $D_{r-k}(\varphi)$ is not empty, see Proposition \ref{b+}.

The next natural question to ask is that if one can determine connectedness of $D_{r-k}(\varphi)$ (as a matter of fact, since degeneracy loci are normal, equivalently their irreducibility). 

We can give an answer for some $k$'s, as follows.

\begin{bthm} \hskip 3cm
\label{connesse}

Let $X$ be a smooth irreducible projective variety of dimension $n$. Let $\E$  be a rank $r$ globally generated bundle on $X$ such that $c_k(\E) \ne 0$ for some integer $k$ and let $s=h^0(\E^*)$, so that $r \ge k+s$. Consider morphisms $\varphi : \O_X^{\oplus (r+1-k)} \to \E$ such that the degeneracy loci $D_{r-k}(\varphi)$ are reduced of pure codimension $k$. Also, consider the following conditions:
\begin{itemize}
\item[(i)] $c_{k+1}(\E)=0$.
\item [(ii)] The degeneracy loci $D_{r-k}(\varphi)$ as above have at least $r+1-k-s$ connected (or irreducible) components.
\item[(iii)] The degeneracy loci $D_{r-k}(\varphi)$ as above are disconnected.
\end{itemize} 
If $k \in \{1, 2\}$ we have: (i) implies (ii), (iii) implies (i), and if $r \ge k+s+1$, then (i), (ii) and (iii) are equivalent. If $k=3$ and $H^1(\O_X)=0$, then (i) implies (ii). 

Moreover, under one of the following conditions, for $k \in \{1, 2\}$, the degeneracy loci $D_{r-k}(\varphi)$ as above are connected:
\begin{itemize}
\item[(iv)] $k \le \min\{r-1,n-1\}$ and $\B_+(\E) \ne X$ (that is $\E$ is V-big).
\item [(v)] $\varphi$ is a general morphism and $D_{r-k}(\varphi)$ is singular.
\end{itemize} 
\end{bthm}

Note that, for any $k$, the number of connected components is constant, as soon as we require that $D_{r-k}(\varphi)$ is reduced of pure codimension $k$, see Proposition \ref{conn}.

We point out that the inequality $r \ge k+s+1$, needed to get disconnectedness in (ii), is sharp, see Remark \ref{bana}. Also observe that, in many cases, we can show that the number of connected components is exactly $r+1-k-s$ (see Lemmas \ref{numcc}, \ref{sharp} and Remark \ref{numcc2}).

We are especially interested in the case when $\E$ is an Ulrich vector bundle on a smooth variety $X \subset \P^N$, namely $\E$ satisfies the vanishings $H^i(\E(-p))=0$ for $1 \le p \le \dim X$. Ulrich bundles have emerged as an important class of vector bundles to be studied (see for example \cite{be, es, cmrpl}). 

In the case $k=2$ when $\E$ is an Ulrich bundle, the degeneracy loci $D_{r-2}(\varphi)$ are known as Ulrich subvarieties (see \cite{lr1} or Definition \ref{usv}) and their existence is equivalent by \cite[Thm.~1]{lr1} to the existence of Ulrich bundles, which is still a widely open and considerably studied problem.

We first give the following classification of Ulrich bundles with empty Ulrich subvarieties. 

\begin{bthm} \hskip 3cm
\label{c2=0}

Let $X \subset \P^N$ be a smooth irreducible variety of dimension $n \ge 2$, degree $d \ge 2$ and let $\E$ be a rank $r \ge 2$ Ulrich bundle on $X$. The following are equivalent:
\begin{enumerate}[(i)]
\item All Ulrich subvarieties associated to $\E$ are empty.
\item There exists an empty Ulrich subvariety associated to $\E$.
\item $c_2(\E)=0$.
\item $(X,\O_X(1),\E)$ is a linear Ulrich triple over a curve (see \cite{lms} or Definition \ref{not4}).
\end{enumerate}
\end{bthm}

Next, as a consequence of Theorem \ref{connesse}, we give a characterization of connectedness Ulrich subvarieties.

\begin{bcor} \hskip 3cm
\label{singconn}

Let $X \subset \P^N$ be a smooth irreducible variety of dimension $n$. Let $\E$  be a rank $r \ge 3$ Ulrich bundle on $X$ with $c_2(\E) \ne 0$. Then the following are equivalent:
\begin{itemize}
\item[(i)] $c_3(\E)=0$.
\item [(ii)] The Ulrich subvarieties associated to $\E$ have at least $r-1$ connected (or irreducible) components.
\item[(iii)] The Ulrich subvarieties associated to $\E$ are disconnected.
\end{itemize} 
Moreover, under the following conditions, Ulrich subvarieties $Z$ associated to $\E$ are connected:
\begin{itemize}
\item[(iv)] $X$ is not covered by lines $L \subset X$ such that $\E_{|L}$ is not ample and $n \ge 3$.
\item [(v)] $Z$ is singular and general (that is $Z=D_{r-2}(\varphi)$, with $\varphi$ a general morphism). 
\end{itemize}
\end{bcor}

In the case of surfaces, Ulrich subvarieties are $0$-dimensional smooth subschemes and we can classify completely the connected cases, which correspond to $c_2(\E)=1$.

\begin{bthm} \hskip 3cm
\label{superficie}

Let $S \subset \P^N$ be a smooth irreducible surface of degree $d \ge 2$ and let $\E$ be an Ulrich bundle on $S$ with $c_2(\E) \ne 0$. Then the Ulrich subvarieties associated to $\E$ are connected if and only if $(S,\O_S(1),\E)$ is one of the following:
\begin{itemize}
\item[(i)] $(Q,\O_Q(1), \mathcal S' \oplus \mathcal S'')$, where $Q$ is a smooth quadric in $\P^3$ and $\mathcal S', \mathcal S''$ are the two spinor bundles on $Q$.
\item [(ii)] $(\Gamma,\O_{\Gamma}(1),\O_{\Gamma}(T)^{\oplus 2})$, where $\Gamma$ is a smooth cubic in $\P^3$ and $T \subset \Gamma$ is a twisted cubic.
\item[(iii)] $(\Sigma,\O_{\Sigma}(1),\O_{\Sigma}(C_0+f)^{\oplus 2})$, where $\Sigma$ is a smooth non-degenerate cubic in $\P^4$.
\end{itemize} 
\end{bthm}

We also study some standard cases on $3$-folds (see Proposition \ref{non big}) and give some connectedness criteria and several examples with vanishing of Chern classes, see sections \ref{sette} and \ref{dieci}. 

Finally, in section \ref{undici}, we prove some results for rank $2$ bundles. We have a characterization when $H^1(\O_X)=0$, see Lemma  \ref{conn2} and several connectedness or disconnectedness criteria, see Lemma  \ref{conn3}. These show that, when $H^1(\O_X) \ne 0$, the connectedness scenery  appears to be kind of unpredictable.

\section{Notation}

Unless otherwise specified, we henceforth establish throughout the paper the following.

\begin{notation} 

\hskip 3cm

\begin{itemize} 
\item $X$ is a smooth irreducible projective variety of dimension $n \ge 1$.
\item $\rho(X)$ is the Picard number of $X$.
\item $\nu(\L)$ is the numerical dimension of a nef line bundle $\L$ on $X$.
\item $A^k(X)$ is the group rational equivalence classes of $(n-k)$-cycles on $X$.

Moreover, when $X \subseteq \P^N$:
\item $H \in |\O_X(1)|$ is a very ample divisor. 
\item $d=H^n$ is the degree of $X$.
\item $X$ is subcanonical if $K_X = -i_X H$ for some $i_X \in \Z$.
\item For $1 \le i \le n-1$, let $H_i \in |H|$ be general divisors and set $X_n:=X$ and $X_i=H_1\cap \ldots \cap H_{n-i}$. 
\end{itemize} 
\end{notation}

We work over the complex numbers.

\section{Definitions and preliminary results}

We will need the following statement on vanishing of cohomology.

\begin{lemma} 
\label{van}
Let $A$ be a divisor on $X$ with $h^0(A) \ge 2$ and let $H=hA$, with $h \ge 1$. Let $\F, \G$ be two vector bundles on $X$. We have:
\begin{itemize}
\item[(i)] If $H^0(\G(2A))=H^1(\G(A))=0$, then $H^1(\G)=0$.
\item[(ii)] If $H^0(\F(-H))=H^1(\F(-2H))=0$, then $H^1(\F(-jA))=0$ for all $j \ge 2h$.
\end{itemize} 
\end{lemma}
\begin{proof} 
Since $h^0(A) \ge 2$, we can choose $B \in |A|$ and then the exact sequence
$$0 \to \O_X \to A \to A_{|B} \to 0$$
implies that $A_{|B}$ is effective. To see (i), observe that the exact sequence
$$0 \to \G(A) \to \G(2A) \to \G(2A)_{|B} \to 0$$
implies that $H^0(\G(2A)_{|B})=0$. In particular we have that $\dim B \ge 1$ and, since $H^0(\G(A)_{|B}) \subseteq H^0(\G(2A)_{|B})=0$, we deduce that $H^0(\G(A)_{|B})=0$. Then, the exact sequence
$$0 \to \G \to \G(A) \to \G(A)_{|B} \to 0$$
implies that $H^1(\G)=0$. This proves (i). We now show (ii) by induction on $j$. If $j=2h$, then $H^1(\F(-jA))=0$ by hypothesis. If $j \ge 2h+1$, set $\G=\F(-jA)$. Then, by induction,
$$H^1(\G(A))=H^1(\F(-(j-1)A))=0.$$
Also, since $-j+2 \le 1-2h \le -h$ we have that 
$$H^0(\G(2A))=H^0(\F((-j+2)A)) \subseteq H^0(\F(-hA))=0.$$ 
Therefore (i) implies that $H^1(\F(-jA))=0$ and this proves (ii). 
\end{proof}

Given a vector bundle $\E$ on $X$, one can measure of the positivity of $\E$ via stable base loci. We recall the relevant definitions (see for example \cite[Def.'s 2.1 and 2.4]{bkkmsu}).
\begin{defi}
The {\it base locus} of $\E$ is $\Bs(\E)=\{x \in X : H^0(\E) \to \E(x) \ \hbox{is not surjective}\}$. The {\it stable base locus} of $\E$ is $\B(\E)=\bigcap_{m>0} \Bs(S^m\E)$. Let $A$ be an ample line bundle on $X$. The {\it augmented base locus} of $\E$ is 
$$\B_+(\E)=\bigcap_{p/q \in \Q^{>0}} \B((S^q \E)(-pA)).$$
The bundle $\E$ is {\it V-big} if $\B_+(\E) \ne X$. 
\end{defi}
The definition of $\B_+(\E)$ does not depend on the choice of $A$ \cite[(2.5.1)]{bkkmsu}. 

A very useful geometrical vision of vector bundles can be given via degeneracy loci.

\begin{defi}
Let $\varphi : \E \to \F$ be a morphism of vector bundles of ranks $e, f$ on $X$. For any $k \in \Z$ with $0 \le k \le \min\{e,f\}$ we denote the $k$-th degeneracy locus of $\varphi$ by $D_k(\varphi)$, that is
$$D_k(\varphi)=\{x \in X : \rk \varphi(x) \le k\}.$$
\end{defi}

\begin{remark}
\label{degeneracy}
Degeneracy loci have a natural scheme structure, given locally by the vanishings of the $(k+1) \times (k+1)$ minors of the matrix defining $\varphi(x)$. Equivalently \cite[\S (2.1)]{las}, the ideal sheaf of $D_k(\varphi)$ is the image of the morphism $\Lambda^{k+1} \E \otimes \Lambda^{k+1} \F^* \to \O_X$ induced by $\Lambda^{k+1} \varphi$.
\end{remark}

In the following lemma we collect some known results.

\begin{lemma}
\label{red}
Let $\E, \F$ be vector bundles on $X$ of ranks $e, f$ respectively and such that $\E^* \otimes \F$ is globally generated. Let $\varphi : \E \to \F$ be a general morphism. Let $k \in \Z$ be such that $0 \le k \le \min\{e, f\}$. If $D_k(\varphi) \ne \emptyset$, then $D_k(\varphi)$ is regular in codimension $e+f-2k$, normal, reduced, Cohen-Macaulay, of pure codimension $(e-k)(f-k)$ and $\Sing(D_k(\varphi))=D_{k-1}(\varphi)$.
\end{lemma} 
\begin{proof}
It follows from \cite[Statement (folklore)(i), \S 4.1]{ba} that $D_k(\varphi)$ is of pure codimension $(e-k)(f-k)$ and $\Sing(D_k(\varphi))=D_{k-1}(\varphi)$. Then, $D_k(\varphi)$ is Cohen-Macaulay by \cite[Ch.~II, Prop.~(4.1)]{acgh}. If $D_k(\varphi)$ is smooth we are done. Otherwise, \cite[Statement (folklore)(i), \S 4.1]{ba} implies that $\Sing(D_{k-1}(\varphi)) \ne \emptyset$ is of pure codimension $(e+1-k)(f+1-k)$. In particular $n-(e-k)(f-k) \ge e+f-2k+1 \ge 2$. Now $\codim_{D_k(\varphi)}D_{k-1}(\varphi)=e+f-2k+1 \ge 2$ so that $D_k(\varphi)$ is regular in codimension $e+f-2k$ and is $R_1$. Moreover, for any $x \in D_k(\varphi)$ and any prime ideal $\mathfrak p$ of $A:=\O_{D_k(\varphi),x}$ we have that $A_{\mathfrak p}$ is Cohen-Macaulay by \cite[Thm.~II.8.21A(b)]{h1}. Hence $\depth A_{\mathfrak p} = \dim A_{\mathfrak p} = n-(e-k)(f-k) \ge 2$, so that $A$ is $S_2$. Then $A$ is normal by \cite[Thm.~II.8.22A]{h1} and reduced by \cite[Lemma 10.153.3, \href{https://stacks.math.columbia.edu/tag/031R}{Tag 031R}]{st}, whence so is $D_k(\varphi)$.
\end{proof}

\section{Generalities on globally generated bundles}

We collect here some simple but useful results about globally generated bundles.

\begin{lemma} 
\label{gg}
Let $\E$ be a rank $r \ge 1$ globally generated bundle on $X$. For any $1 \le i \le n$, let $H$ be a very ample divisor on $X$. We have:
\begin{itemize}
\item[(i)] $c_1(\E)^i=0$ if and only if $c_1(\E)^iH^{n-i}=0$. 
\item [(ii)] $c_i(\E)=0$ if and only if $c_i(\E)H^{n-i}=0$. 
\end{itemize} 
Moreover, let $k \in \Z$ be such that $1 \le k \le \min\{r, n\}$. Let $\varphi : \O_X^{\oplus (r+1-k)} \to \E$ be a general morphism. Then the following are equivalent:
\begin{itemize}
\item[(iii)] $D_{r-k}(\varphi)=\emptyset$.
\item[(iv)] $c_i(\E)=0$ for $i \ge k$.
\item[(v)] $c_k(\E)=0$. 
\end{itemize}
Also, if any of (iii)-(v) does not hold, then $D_{r-k}(\varphi)$ has pure codimension $k$ and
$c_k(\E)=[D_{r-k}(\varphi)] \in A^k(X)$.
\end{lemma}
\begin{proof}
To see (i), one implication being obvious, assume that $c_1(\E)^iH^{n-i}=0$. Since $\E$ is globally generated, it follows that $\det \E$ is globally generated and $(\det \E)^iH^{n-i}=0$ implies that $c_1(\E)^i=0$ by \cite[Cor.~3.15]{fl2} (see also \cite[Prop.~3.7]{fl1}). This proves (i). Suppose now that $D_{r-k}(\varphi) \ne \emptyset$. Then $D_{r-k}(\varphi)$ has pure dimension $n-k \ge 0$ by Lemma \ref{red}, so that $c_k(\E)=[D_{r-k}(\varphi)] \in A^k(X)$. This proves the last assertion of the lemma, once proved the equivalence (iii)-(v), that we now show. If (iii) holds, we have that $\varphi$ has constant rank $r+1-k$ and we get an exact sequence
$$0 \to \O_X^{\oplus (r+1-k)} \to \E \to \F \to 0$$
where $\F$ is also a vector bundle, of rank $k-1$. But then, for any $i \ge k$ we have
$$c_i(\E)=\sum\limits_{j=0}^i c_j(\O_X^{\oplus (r+1-k)})c_{i-j}(\F)=c_i(\F)=0.$$ 
Hence (iii) implies (iv). Clearly (iv) implies (v). Now assume (v). If $D_{r-k}(\varphi) \ne \emptyset$ we get a contradiction both if $n-k=0$, since $0=c_k(\E)=[D_{r-k}(\varphi)]$, while $D_{r-k}(\varphi)$ is $0$-dimensional and nonempty and if $n-k>0$, since we would have that  $0 < H^kD_{r-k}(\varphi)=H^kc_k(\E)=0$. This proves that (v) implies (iii), hence the equivalence (iii)-(v) and also proves the last assertion of the lemma. It remains to prove (ii). One implication being obvious, we assume that $c_i(\E)H^{n-i}=0$. If $i=n$ we are done, so assume that $i \le n-1$. Setting $k=i$, if $c_i(\E) \ne 0$, we get that $D_{r-i}(\varphi)$ is nonempty of pure codimension $i$ and $c_i(\E)=[D_{r-i}(\varphi)]$, giving the contradiction $0 < H^{n-i}D_{r-i}(\varphi)=H^{n-i}c_i(\E) =0$. This proves (ii) and concludes the proof of the lemma. 
\end{proof}

In the sequel, we will be interested in the vanishing of some Chern classes, which we now study.

\begin{lemma} 
\label{c3no}
Let $\E$ be a rank $r$ globally generated bundle on $X \subset \P^N$. We have:
\begin{itemize}
\item [(i)] If $c_1(\E)^t = 0$ for some $t \ge 1$, then $c_i(\E)=0$ for all $i \ge t$.
\item [(ii)] If $t \ge 1$ and $c_1(\E)^t \ne 0$ (in particular if $c_t(\E) \ne 0$), then $H^j(-\det \E)=0$ for $0 \le j \le t-1$.
\item [(iii)] Let $s=h^0(\E^*)$. If $s \ge 1$, then $\E \cong \O_X^{\oplus s} \oplus \E_1$ with $\E_1$ globally generated of rank $r-s$ and $h^0(\E_1^*)=0$. Moreover, if $c_k(\E) \ne 0$, then $r \ge k+s$. 
\end{itemize}
\end{lemma}
\begin{proof}
Assume that $c_1(\E)^t = 0$. The conclusion (i) being obvious if $n \le t-1$, suppose that $n \ge t$. Then $c_1(\E_{|X_t})^t=c_1(\E)^tH^{n-t} = 0$, hence $c_t(\E)H^{n-t}=c_t(\E_{|X_t})=0$ by \cite[Cor.~2.7]{dps}. Therefore $c_t(\E) = 0$ by Lemma \ref{gg}(ii) and then (i) holds by the same lemma. To see (ii), first observe that if $c_t(\E) \ne 0$, then $c_1(\E)^t \ne 0$ by (i). Hence, to prove (ii), we can assume that $c_1(\E)^t \ne 0$, and we have, in particular, that $n \ge t$. Set $\det \E = \O_X(D)$. We prove that $H^j(\O_X(-D))=0$ for $0 \le j \le t-1$ by induction on $n$. If $n=t$, we know that $D^t>0$, hence $D$ is big and nef and therefore $H^j(\O_X(-D))=0$ by Kawamata-Viehweg's vanishing theorem. If $n \ge t+1$, note that $c_1(\E_{|X_{n-1}})^t=c_1(\E)^tH \ne 0$, for otherwise also $c_1(\E)^tH^{n-t}=0$, implying, by Lemma \ref{gg}(ii), the contradiction $c_1(\E)^t=0$. Then, in the exact sequence
$$0 \to \O_X(-D-H) \to \O_X(-D) \to \O_{X_{n-1}}(-D) \to 0$$
we have that $H^j(\O_{X_{n-1}}(-D))=H^j(-\det (\E_{|X_{n-1}}))=0$ by the inductive hypothesis and also $H^j(\O_X(-D-H))=0$ by Kodaira vanishing. Hence we find that $H^j(\O_X(-D))=0$, proving (ii). Finally, the first part of (iii) is well-known (see for example \cite[Proof of Lemma 3]{su} or use \cite[Lemma 3.9]{o2} and induction). Now, when $c_k(\E) \ne 0$ we have that $c_k(\E_1)=c_k(\E) \ne 0$ and therefore $r-s \ge k$.
\end{proof}

We observe that V-bigness implies non-vanishing of Chern classes, hence also non-emptiness and connectedness (via Theorem \ref{connesse}).

\begin{prop}
\label{b+}
Let $k$ be an integer such that $1 \le k \le \min\{r,n\}$. Let $\E$ be a rank $r$ vector bundle on $X$ with $\B_+(\E) \ne X$. We have:
\begin{itemize}
\item[(i)] If $\E$ is nef, then $c_k(\E) \ne 0$.
\item[(ii)] If $\E$ is globally generated, then $H^0(\E^*)=0$.
\item[(iii)] If $\E$ is globally generated, then $D_{r-k}(\varphi) \ne \emptyset$ for any injective morphism $\varphi : \O_X^{\oplus (r+1-k)} \to \E$.
\end{itemize}
\end{prop}
\begin{proof}
We will use $\Q$-twisted bundles and their Chern classes, see \cite[\S 6.2, 8.1, 8.2]{la2}. We first observe that if $\F \langle \delta \rangle$ is a $\Q$-twisted nef bundle on $X$, then $c_k(\F \langle \delta \rangle)$ is nef. Indeed, this follows from \cite[Proof of Prop.~2.1 and Cor.~2.2]{dps}, since the proof works for $\R$-twisted nef vector bundles. Alternatively, for any $k$-dimensional irreducible subvariety $Y \subset X$, the restriction $(\F \langle \delta \rangle)_{|Y}$ is still nef by \cite[Thm.~6.2.12(v)]{la2}. Then, it is enough to apply \cite[Thm.~8.2.1]{la2} to obtain that $c_k(\F \langle \delta \rangle)\cdot Y = c_k((\F \langle \delta \rangle)_{|Y}) \ge 0$, as desired. We now use Seshadri constants $\varepsilon(\E; x)$, in particular \cite[Rmk.~3.10(c)]{fm}. To see (i), let $x \not\in \B_+(\E)$ and let $\mu : \widetilde X \to X$ be the blowing up of $X$ at $x$, with exceptional divisor $E$. It follows from \cite[Prop.~6.9 and Rmk.~3.10(c)]{fm} that  $\varepsilon(\E; x)>0$, hence it is easy to see that we can find $t \in \Q_{>0}$ such that the $\Q$-twisted vector bundle $(\mu^\ast\E)\langle -tE \rangle$ is nef. If follows from the observation above, that $c_k((\mu^\ast\E) \langle -tE \rangle)$ is nef. Let $Z$ be any irreducible subvariety with $x \in Z \subseteq X, \dim Z=k$ and consider its strict transform $\widetilde Z$ on $\widetilde X.$ Then, as in \cite[Proof of Thm.~7.2]{lo}, we have 
$$0 \le c_k((\mu^\ast \E) \langle -tE \rangle) \cdot \widetilde Z = c_k(\E) \cdot Z-t^k\binom{r}{k} \mathrm{mult}_x(Z)$$ so that $c_k(\E) \ne 0$ and (i) is proved. To see (ii), assume that $s=h^0(\E^*) \ge 1$. By Lemma \ref{c3no}(iii) we have $\E=\O_X^{\oplus s} \oplus \E_1$, with $\E_1$ a globally generated bundle. Since $\E$ is nef and $\B_+(\E)\ne X$, it follows  by \cite[Prop.~6.9]{fm} that there is $x \in X$ such that $\varepsilon(\E; x)>0$. However, as $\varepsilon(\E_1; x) \ge 0$ by the nefness of $\E_1$ \cite[Rmk. 3.10(a)]{fm}, we get by \cite[Lemma 3.31]{fm} that 
$$0 < \varepsilon(\E; x) = \min\{\varepsilon(\O_X^{\oplus s}; x), \varepsilon(\E_1; x)\} = \varepsilon(\O_X^{\oplus s}; x)=0$$ 
a contradiction. This proves (ii). As for (iii), assume that $D_{r-k}(\varphi)= \emptyset$. If $k=1$ then $\varphi$ is an isomorphism and we get that $c_1(\E)=0$, contradicting (i). If $k \ge 2$, we have an exact sequence
$$0 \to \O_X^{\oplus (r+1-k)} \to \E \to \F \to 0$$
where $\F$ is a rank $k-1$ bundle on $X$, hence $c_k(\E)=c_k(\F)=0$, contradicting (i). This proves (iii).
\end{proof}

We add a calculation of the $n$-th Segre class. This will be useful to detect bigness in some examples.

\begin{lemma}
\label{seg}
Let $n \ge 2$ and let $\E$ be a globally generated bundle on $X$ such that $c_2(\E)^2=0$ and $c_3(\E)=0$. Then
$$s_n(\E^*)=c_1(\E)^n-(n-1)c_1(\E)^{n-2}c_2(\E).$$
In particular $\E$ is big if and only if $c_1(\E)^n>(n-1)c_1(\E)^{n-2}c_2(\E)$.
\end{lemma}
\begin{proof}
For $n \ge 2$, consider the $n \times n$ determinant
$$P_n(x_1,x_2)=
\begin{vmatrix}
x_1 & x_2 & 0 & 0 & \dots & 0 & 0 & 0 \\ 
1 & x_1 & x_2 & 0& \dots & 0 & 0 & 0 \\
0 & 1 & x_1 & x_2 & \dots & 0 & 0 & 0 \\
\hdotsfor{8} \\
0 & 0 & 0 & 0 & \dots & 1 & x_1 & x_2 \\
0 & 0 & 0 & 0 & \dots & 0 & 1 & x_1
\end{vmatrix}.
$$
It is easily seen that $P_2(x_1,x_2)=x_1^2-x_2, P_3(x_1,x_2)=x_1^3-2x_1x_2$ and that, for $n \ge 4$,
\begin{equation}
\label{pol}
P_n(x_1,x_2)=x_1P_{n-1}(x_1,x_2)-x_2P_{n-2}(x_1,x_2).
\end{equation}
Setting $c_i=c_i(\E)$ and using \eqref{pol}, it follows by induction on $n \ge 2$ that 
$$P_n(c_1,c_2)=c_1^n-(n-1)c_1^{n-2}c_2.$$
On the other hand, we have that $c_i(\E)=0$ for $i \ge 3$ by Lemma \ref{gg}, hence $P_n(c_1,c_2)$ is just the Schur polynomial of $\E$ associated to the partition $1^n=(1,\ldots,1)$ and therefore (see for example \cite[Exa.~8.3.5]{la2}), we get that 
$$s_n(\E^*)=s_{1^n}(\E)=P_n(c_1,c_2)=c_1^n-(n-1)c_1^{n-2}c_2.$$
As is well-known (see for example \cite[Rmk.~2.2]{lm}), $\E$ is big if and only if $s_n(\E^*)>0$, thus if and only if $c_1(\E)^n>(n-1)c_1(\E)^{n-2}c_2(\E)$.
\end{proof}

The following lemma, relating some degeneracy loci (see also \cite[Rmk.~3.4]{cfk}), will be crucial in the proof of Theorem \ref{connesse}.

\begin{lemma} 
\label{zeta1}
Let $n \ge 2$ and let $\E$ be a rank $r$ globally generated bundle on $X$ such that $c_2(\E) \ne 0$. Let $V \subset H^0(\E)$ be a general subspace with $\dim V = r-1$, let $\varphi : V \otimes \O_X  \to \E$, so that $Z = D_{r-2}(\varphi) \ne \emptyset$. Then there is a normal irreducible Cartier divisor $Y \in |\det(\E)|$, smooth if $n \le 3$, such that $Z \subset Y$. Moreover, if $Y$ is smooth, we have an exact sequence
\begin{equation}
\label{incl2}
0 \to \E^* \to \O_X^{\oplus r} \to \O_Y(Z) \to 0
\end{equation}
and $\O_Y(Z)$ is globally generated by at most $r$ sections. Also, using, on the right hand side, the intersection product on $Y$ and a very ample divisor $H$, we have, if $n \ge 3$, that 
\begin{equation}
\label{c33}
c_3(\E)H^{n-3}=Z^2H_{|Y}^{n-3}.
\end{equation}
\end{lemma}
\begin{proof} 
Note that $Z \ne \emptyset$ by Lemma \ref{gg}, so that $Z$ is reduced of pure codimension $2$ by Lemma \ref{red}. Moreover, if $Z$ is singular, then $\Sing(Z)=D_{r-3}(\varphi)$. Also, the Eagon-Northcott complex gives is an exact sequence
\begin{equation}
\label{succ}
0 \to V \otimes \O_X \to \E \to \I_{Z/X}(D) \to 0.
\end{equation}
Choose a general subspace $V' \subset H^0(\E)$ such that $\dim V' = r$ and $V \subset V'$. Thus, we get a general morphism $\varphi' : V' \otimes \O_X  \to \E$ and setting $Y=D_{r-1}(\varphi')$ we see that $Z \subset Y$: If $x \in Z$, then $\rk \varphi(x) \le r-2$, hence $\rk \varphi'(x) \le r-1$, so that $x \in Y$. In particular $Y$ is nonempty and Lemma \ref{red} gives that $Y$ is smooth if $n \le 3$. Note that $c_1(\E)^2 \ne 0$, for otherwise $c_2(\E)=0$ by Lemma \ref{c3no}(i). Therefore the morphism $\varphi_{\det \E} : X \to \P H^0(\det \E)$ is not composed with a pencil and $Y \in |\det(\E)|$ is connected by Bertini's theorem. Since $Y$ is normal by Lemma \ref{red}, it is irreducible. Now assume that $Y$ is smooth, so that $D_{r-2}(\varphi')=\emptyset$ by Lemma \ref{red}. Moreover also $Z$ is smooth, because, as above, we have that $D_{r-3}(\varphi) \subset D_{r-2}(\varphi')$. Using \eqref{succ}, we have a commutative diagram
$$\xymatrix{& 0 \ar[d] & 0 \ar[d] & & \\ 0 \ar[r] & V \otimes \O_X \ar[d] \ar[r]^{\hskip .4cm \varphi} & \E \ar[d] \ar[r] & \I_{Z/X}(D)  \ar[r] \ar[d] & 0 & \\ 0 \ar[r] & V' \otimes \O_X \ar[r]^{\hskip .4cm \varphi'} \ar[d] & \E \ar[r] \ar[d] & \L \ar[r] \ar[d] & 0 &  \\ & \O_X \ar[d] & 0 & 0 & \\ & 0 & & & }$$
where $\L$ is a sheaf supported on $Y$. The diagram shows that $\L$ is a line bundle on $Y$, since $\rk \varphi'(y)=r-1$ for every $y \in Y$. Also, we get an exact sequence
$$0 \to \O_X \to \I_{Z/X}(D) \to \L \to 0$$
and the map $\O_X \to \I_{Z/X}(D)$ is given by the section defining $Y$. Therefore $\L \cong \I_{Z/Y}(D)$. We will now use the well-known fact that the proof of \cite[Lemma III.7.4]{h1} works also for the sheaf version with $\SHom$ and $\SExt$. Using it, we get 
$$\SExt^1_{\O_X}(\L,\O_X) \cong \SExt^1_{\O_X}( \L,\omega_X)(-K_X) \cong \SHom_{\O_Y}( \L,\omega_Y)(-K_X) \cong \O_Y(K_Y+Z-D)(-K_X) \cong \O_Y(Z).$$ 
Hence, dualizing the exact sequence, obtained in the diagram above,
$$0 \to V' \otimes \O_X \to \E \to \L \to 0$$
we get the exact sequence
$$0 \to \E^* \to \O_X^{\oplus r} \to \O_Y(Z) \to 0$$
showing \eqref{incl2} and that $\O_Y(Z)$ is globally generated by at most $r$ sections. Finally, to see \eqref{c33}, if $n=3$, we set $X'=X, \E'=\E$ and $Z'=Z$. If $n \ge 4$, cutting down with $n-3$ general $H_1, \ldots, H_{n-3} \in |H|$, we get a smooth $3$-fold $X'=X_3$ and a globally generated bundle $\E'=\E_{|X_3}$. Setting $Z'=Z \cap H_1 \cap \ldots \cap H_{n-3}$, we have that 
$$Z' = Z \cap X'=D_{r-2}(\varphi) \cap X' =D_{r-2}(\varphi_{|X'})$$
so that, in particular, $[Z']=c_2(\E')$. Moreover, setting $Y'=Y \cap H_1 \cap \ldots \cap H_{n-3}$, we see that $Y' \in |\det \E'|$ is a smooth irreducible surface containing $Z'$ and \eqref{incl2} gives an exact sequence
\begin{equation}
\label{ze2}
0 \to (\E')^* \to \O_{X'}^{\oplus r} \to \O_{Y'}(Z') \to 0.
\end{equation}
We compute the Euler characteristics in \eqref{ze2}. To this end, we set $c_i=c_i(X')$ and $d_i=c_i(\E')$, for $1 \le i \le 3$. Now, Riemann-Roch on $X'$ gives
\begin{equation}
\label{pri2}
\chi((\E')^*)=r\chi(\O_{X'})-\frac{1}{12}d_1(c_1^2+c_2)+\frac{1}{4}c_1(d_1^2-2d_2)-\frac{1}{6}(d_1^3-3d_1d_2+3d_3).
\end{equation}
The exact sequence
\begin{equation}
\label{sec22}
0 \to \O_{X'}(-Y') \to \O_{X'} \to \O_{Y'} \to 0
\end{equation}
gives, using Riemann-Roch on $X'$,
\begin{equation}
\label{sec2}
\chi(\O_{Y'})=\chi(\O_{X'})-\chi(\O_{X'}(-Y'))=\frac{1}{12}d_1(d_1-c_1)(2d_1-c_1)+\frac{1}{12}d_1c_2.
\end{equation}
Moreover, by Riemann-Roch on $Y'$, we get, using adjunction, \eqref{sec2} and $[Z']=c_2(\E')$,
\begin{equation}
\label{ter2}
\begin{aligned}[t] 
\chi(\O_{Y'}(Z')) & = \chi(\O_{Y'})+\frac{1}{2}Z'(Z'-K_{Y'})=\chi(\O_{Y'})+\frac{1}{2}d_2(c_1-d_1)+\frac{1}{2}(Z')^2= \\
& = \frac{1}{12}d_1(d_1-c_1)(2d_1-c_1)+\frac{1}{12}d_1c_2+\frac{1}{2}d_2(c_1-d_1)+\frac{1}{2}(Z')^2.
\end{aligned}
\end{equation}
Therefore \eqref{ze2}, together with \eqref{pri2} and \eqref{ter2}, gives
$$\begin{aligned}[t] 
r\chi(\O_{X'}) & =\chi((\E')^*)+\chi(\O_{Y'}(Z'))= r\chi(\O_{X'})-\frac{1}{12}d_1(c_1^2+c_2)+\frac{1}{4}c_1(d_1^2-2d_2)-\frac{1}{6}(d_1^3-3d_1d_2+3d_3)+\\
& + \frac{1}{12}d_1(d_1-c_1)(2d_1-c_1)+\frac{1}{12}d_1c_2+\frac{1}{2}d_2(c_1-d_1)+\frac{1}{2}(Z')^2= r\chi(\O_{X'})-\frac{1}{2}d_3+\frac{1}{2}(Z')^2.
\end{aligned}$$
Hence $c_3(\E)H^{n-3}=c_3(\E')=d_3=(Z')^2=Z^2H_{|Y}^{n-3}$ and \eqref{c33} is proved. 
\end{proof}

We now give a fact that will be useful to study Ulrich subvarieties in the case of a linear Ulrich triple.

\begin{lemma}
\label{usub}
Let $X, B$ be two projective varieties with $X$ smooth and let $\pi : X \to B$ be a flat morphism with $\pi_*\O_X\cong\O_B$. Let $\E_B$ be a rank $r$ bundle on $B$ such that $\E=\pi^*\E_B$ is globally generated. Let $k \in \Z$ be such that $1 \le k \le r$, let $V \subseteq H^0(\E)$ be a general subspace of dimension $r+1-k$ and, if $\varphi : V \otimes \O_X \to \E$, assume that $Z:=D_{r-k}(\varphi)$ is nonempty. Then there exists $V_B \subset H^0(\E_B)$ a general subspace such that $V=\pi^* V_B$ and, if $\varphi_B : V_B \otimes \O_B \to \E_B$ is the associated morphism and $Z_B=D_{r-k}(\varphi_B)$, we have that  $\O_Z \cong \pi^*\O_{Z_B}$, $Z$ is the scheme-theoretic inverse image of $Z_B$ under $\pi$ and $Z_B$ has pure codimension $k$ in $B$.
\end{lemma}
\begin{proof}
$\pi$ is surjective being both open and closed. Note that $B$ is smooth and $\E_B$ is globally generated since $X$ and $\E=\pi^*\E_B$ respectively are (see for example \cite[Rmk.~4.3.25 and Exc.~5.1.29(b)]{liu}). Then, by projection formula, we obtain  
$$H^0(\E)\cong H^0(\E_B \otimes \pi_\ast\O_X)\cong H^0(\E_B).$$ Moreover, as the pull-back of sections $\pi^\ast\colon H^0(\E_B)\to H^0(\E)$ is injective by \cite[Cor.~2.2.8]{g1}, the above isomorphism is induced by $\pi^*$. We set $V_B \subseteq H^0(\E_B)$ to be the subspace such that $V=\pi^* V_B$. Then $V_B \subseteq H^0(\E_B)$ is general and we get that $Z_B \ne \emptyset$: In fact, if not, we would have that $c_k(\E_B)=0$ by Lemma \ref{gg}, and therefore also $c_k(\E)=\pi^*c_k(\E_B)=0$, giving, by Lemma \ref{gg} again, the contradiction $Z = \emptyset$. Therefore $Z_B$ has pure codimension $k$ in $B$ by Lemma \ref{red}. It follows that we have the Eagon-Northcott resolution
$$0 \to S^{k-1}V_B \otimes \O_B \to \cdots \to V_B \otimes \Lambda^{k-2} \E_B \to \Lambda^{k-1}\E_B \to \I_{Z_B/B}(\det \E_B) \to 0$$
whose pull-back via $\pi$ is, since $V \cong V_B$ via pull-back of sections, the exact sequence
$$0 \to S^{k-1}V \otimes \O_X \to \cdots \to V \otimes \Lambda^{k-2} \E \to \Lambda^{k-1} \E \to \pi^\ast(\I_{Z_B/B}(\det \E_B)) \to 0.$$
On the other hand, we have the analogous resolution
$$0 \to S^{k-1}V \otimes \O_X \to \cdots \to V \otimes \Lambda^{k-2} \E \to \Lambda^{k-1} \E \to  \I_{Z/X}(D) \to 0$$
where $D = \det \E=\pi^*(\det \E_B)$. We get that $\I_{Z/X}(D) \cong \pi^*(\I_{Z_B/B}(\det \E_B))$ and therefore 
\begin{equation}
\label{idp}
\I_{Z/X} \cong \pi^*\I_{Z_B/B}.
\end{equation} 
Now, pulling back the exact sequence
$$0 \to \I_{Z_B/B} \to \O_B \to \O_{Z_B} \to 0$$
we get an exact sequence
$$0 \to \pi^*\I_{Z_B/B} \to \O_X \to \pi^*\O_{Z_B} \to 0$$
and \eqref{idp} shows that $\O_Z \cong \pi^*\O_{Z_B}$. Since  $\pi$ is flat, we also have (see for example \cite[Lemma 26.4.7, \href{https://stacks.math.columbia.edu/tag/01HQ}{Tag 01HQ}]{st}) that $Z$ is the scheme-theoretic inverse image of $Z_B$.
\end{proof}

\section{Generalities on Ulrich vector bundles}

We will often use the following well-known properties of Ulrich bundles

\begin{lemma}
\label{ulr}
Let $\E$ be a rank $r$ Ulrich bundle on $X \subset \P^N$. We have:
\begin{itemize}
\item[(i)] $\E$ is globally generated. 
\item [(ii)] $\E_{|X_{n-1}}$ is Ulrich on a smooth hyperplane section $X_{n-1}$ of $X$.
\item[(iii)] $\det \E$ is globally generated and it is not trivial, unless $(X,\O_X(1), \E) = (\P^n, \O_{\P^n}(1), \O_{\P^n}^{\oplus r})$.
\item[(iv)] $H^0(\E^*)=0$, unless $(X,\O_X(1),\E) = (\P^n, \O_{\P^n}(1), \O_{\P^n}^{\oplus r})$.
\item [(v)] $\O_X(l)$ is Ulrich if and only if $(X,\O_X(1),l)=(\P^n,\O_{\P^n}(1),0)$.
\item [(vi)] $\E$ is aCM.
\item [(vii)] $h^0(\E)=rd$.
\item [(viii)] If $(X,\O_X(1))=(\P^n, \O_{\P^n}(1))$, then $\E=\O_{\P^n}^{\oplus r}$.
\item [(ix)] $\E^*(K_X+(n+1)H)$ is Ulrich.
\item [(x)] $c_1(\E) H^{n-1}=\frac{r}{2}[K_X+(n+1)H] H^{n-1}$.
\end{itemize} 
\end{lemma}
\begin{proof}
For (i)-(vi) and (x), see for example \cite[Lemma 3.1]{lr1}. For (vii) see \cite[Prop.~2.1]{es} (or \cite[Thm.~2.3]{be}) and for (viii) see \cite[(3.1)]{be}. (ix) follows by definition of Ulrich and Serre's duality.
\end{proof}

We also recall the following examples of Ulrich bundles.

\begin{defi}
\label{spin}
For $n \ge 2$, we let $Q_n \subset \P^{n+1}$ be a smooth quadric. We let $S$ ($n$ odd), and $S', S''$ ($n$ even), be the vector bundles on $Q_n$, as defined in \cite[Def.~1.3]{o}. The {\it spinor bundles} on $Q_n$ are ${\mathcal S}={\mathcal S}_n = S(1)$ if $n$ is odd and ${\mathcal S}'={\mathcal S}'_n = S'(1)$, ${\mathcal S}''={\mathcal S}''_n = S''(1)$, if $n$ is even. 
\end{defi}

\begin{defi}
\label{not4}
Let $\E$ be a vector bundle on $X \subset \P^N$. We say that $(X,\O_X(1),\E)$ is a {\it linear Ulrich triple} if there are a smooth irreducible variety $B$ of dimension $b \ge 1$, a very ample vector bundle $\F$ and a vector bundle $\G$ on $B$ such that $(X,\O_X(1),\E)=(\P(\F), \O_{\P(\F)}(1), \pi^*(\G(\det \F)))$, where $\pi: X \cong \P(\F) \to B$ is the bundle map and $H^j(\G \otimes S^k \F^*)=0 \ \hbox{for all} \ j \ge 0, 0 \le k \le b-1$.
\end{defi}
When $(X,\O_X(1),\E)$ is a linear Ulrich triple, then $\E$ is an Ulrich bundle on $X$ by \cite[Lemma 4.1]{lo}. 

Next, we recall some definitions and facts in \cite{lr1}.

\begin{defi}
\label{usv}
Let $n \ge 2$ and $d \ge 2$. Let $\E$ be a rank $r \ge 2$ Ulrich bundle on $X \subset \P^N$. Let $V \subset H^0(\E)$ be a subspace of dimension $r-1$ such that, if $\varphi : V \otimes \O_X \to \E$ is the associated morphism and $Z=D_{r-2}(\varphi)$, then $Z$ satisfies the following conditions (in particular these hold, by Lemma \ref{red}, if $V$ is a general subspace of $H^0(\E)$):
\begin{itemize}
\item[(a)] $Z$ is either empty or of pure codimension $2$,
\item[(b)] if $Z \ne \emptyset$ and either $r=2$ or $n \le 5$, then $Z$ is smooth (possibly disconnected),
\item[(c)] if $Z \ne \emptyset$ and $n \ge 6$, then $Z$ is either smooth or is normal, Cohen-Macaulay, reduced and with $\dim \Sing(Z) = n-6$.
\end{itemize}
We call $Z$ {\it an Ulrich subvariety associated to $\E$}. We say that $Z$ is {\it a general Ulrich subvariety associated to $\E$} if $V$ is a general subspace of $H^0(\E)$.
\end{defi}

\begin{remark}
\label{usv2}
Let $Z \subset X$ be any subvariety satisfying (a)-(c) above and (i)-(vi) of \cite[Thm.~1]{lr1}. It follows from \cite[Thm.~1]{lr1} that there is an Ulrich bundle $\E$ such that $Z$ is associated to $\E$.
\end{remark}

We now prove Theorem \ref{c2=0}.

\renewcommand{\proofname}{Proof of Theorem \ref{c2=0}}
\begin{proof}
Ulrich subvarieties exist by \cite[Thm.~1]{lr1}, therefore (i) implies (ii). If (ii) holds, there is an empty Ulrich subvariety associated to $\E$, hence $c_2(\E)=0$ by Lemma \ref{gg}, that is (iii). Next, assume (iii). Let $S=X_2, \E'=\E_{|S}, D' \in |\det \E'|$. We have that $\E'$ is a rank $r$ Ulrich bundle on $S$ by Lemma \ref{ulr}(ii) and $\E'$ is globally generated by Lemma \ref{ulr}(i). Also, $c_2(\E')=c_2(\E)_{|S}=c_2(\E) \cdot H^{n-2}=0$. Choosing $r-1$ general sections in $H^0(\E')$ we get a general morphism $\varphi: \O_S^{\oplus (r-1)} \to \E'$ and $D_{r-2}(\varphi)=\emptyset$ by Lemma \ref{gg} with $k=2$. It follows that we have an exact sequence
\begin{equation}
\label{es}
0 \to \O_S^{\oplus (r-1)} \to \E' \to \O_S(D') \to 0.
\end{equation}
Now 
\begin{equation}
\label{ext}
\Ext^1(\O_S(D')),\O_S^{\oplus (r-1)}) \cong H^1(\O_S(-D'))^{\oplus (r-1)}.
\end{equation}
By Lemma \ref{ulr}(iii) we have that $\O_S(D')$ is globally generated, hence nef. If $(D')^2 > 0$, then $D'$ is big and Kawamata-Viehweg's vanishing theorem shows that $H^1(\O_S(-D')))=0$. Hence \eqref{ext} implies that \eqref{es} splits, $\E' \cong \O_S^{\oplus (r-1)} \oplus \O_S(D')$ and therefore $\O_S$ is Ulrich. It follows from Lemma \ref{ulr}(v) that $(S,\O_S(1)) \cong (\P^2, \O_{\P^2}(1))$, hence $d=1$, a contradiction. Therefore $(D')^2=0$. Then $c_1(\E)^2 \cdot H^{n-2}=(D')^2=0$, so that $c_1(\E)^2=0$ by Lemma \ref{gg}(i). Note that $c_1(\E) \ne 0$, for otherwise $d=1$ by \cite[Lemma 2.1]{lo}. Hence $\nu(\det(\E))=1$ and, if $\Phi : X \to {\mathbb G}(r-1, \P H^0(\E))$ and $F_x = \Phi^{-1}(\Phi(x))$, then $\dim F_x = n-1$ for every $x \in X$ by \cite[Thm.~2]{ls}. Therefore $(X,\O_X(1),\E)$ is as in (iv) by \cite[Lemmas 2.10 and 2.12]{lms}.

Finally, if $(X,\O_X(1),\E)$ is as in (iv), then $c_2(\E)=0$ and therefore, if $Z$ is any Ulrich subvariety associated to $\E$, we have that $Z=D_{r-2}(\varphi)$ for some morphism $\varphi : V \otimes \O_X \to \E$ by Remark \ref{usv}. Hence $Z=\emptyset$ by Lemma \ref{gg}. Thus (iv) implies (i) and we are done.
\end{proof}
\renewcommand{\proofname}{Proof}

We observe the following simple consequence of Theorem \ref{c2=0}.

\begin{remark}
\label{usv3}
Let $X \subset \P^N$ be a smooth irreducible variety of dimension $n \ge 2$, degree $d \ge 2$ and let $\E$ be a rank $r \ge 2$ Ulrich bundle on $X$. If $c_1(\E)^2 \ne 0$ (hence, in particular, if $\E$ is big), then all Ulrich subvarieties associated to $\E$ are nonempty.

Indeed, note that if $\E$ is big, then $c_1(\E)^n > 0$ by \cite[Rmk.~2.2]{lm}), hence also $c_1(\E)^2 \ne 0$. Now, if there is an  empty Ulrich subvariety associated to $\E$, then Theorem \ref{c2=0} implies that $(X,\O_X(1),\E)$ is a linear Ulrich triple over a curve, hence $c_1(\E)^2=0$, a contradiction.
\end{remark}

\section{Connectedness of degeneracy loci and of Ulrich subvarieties}
\label{sette}

We study in this section the connectedness of some degeneracy loci associated to a given globally generated bundle. As a particular case, we will get connectedness of Ulrich subvarieties. 

Our first observation is that they all have the same number of connected components.

\begin{prop}
\label{conn}
Let $\E$ be a rank $r$ globally generated bundle on $X$, let $k \ge 1$ and assume that $c_k(\E) \ne 0$. Consider the set of morphisms $\varphi : \O_X^{\oplus (r+1-k)} \to \E$ such that the degeneracy locus $D_{r-k}(\varphi)$ is reduced of pure codimension $k$. Then the above set is open and all such degeneracy loci $D_{r-k}(\varphi)$ have the same number of connected (or irreducible) components. In particular, if $\E$ is Ulrich with $c_2(\E) \ne 0$, then all Ulrich subvarieties associated to $\E$ have the same number of connected components.
\end{prop}
\begin{proof}
Let $W=\Hom_{\O_X}(\O_X^{\oplus (r+1-k)},\E)$, let $T=\Spec(\C[W^*])$ and consider $X \times T$ with projections $\pi_1 : X \times T \to X, \pi_2 : X \times T \to T$. A (closed) point $t \in T$ corresponds to a morphism $\varphi_t : \O_X^{\oplus (r+1-k)} \to \E$. As is well-known (see for instance the proof of \cite[Statement (folklore)(i), \S 4.1]{ba}), the degeneracy loci $D_{r-k}(\varphi_t)$, for $\varphi_t \in W$, arise as fibers of a morphism $\pi_{\mathcal Z} : {\mathcal Z} \to T$, where ${\mathcal Z} \subset X \times T$ is a certain closed subscheme of codimension $k$ and $\pi_{\mathcal Z}={\pi_2}_{| {\mathcal Z}}$. Since $X \to \Spec(\C)$ is proper, then so is $\pi_2$ by base extension. In particular $\pi_{\mathcal Z}$ is proper as well. We now observe that $\pi_{\mathcal Z}$ is surjective. Indeed, we know from Lemmas \ref{gg} and \ref{red}, that we can find an open subset $U' \subset T$ such that $Z_t:=D_{r-k}(\varphi_t)\subset X$, for $t \in U'$, is reduced of pure codimension $k$. This means that $U' \subset \pi_{\mathcal Z}({\mathcal Z})$. As $T$ is integral and $\pi_{\mathcal Z}$ is closed, we conclude that $\pi_{\mathcal Z}({\mathcal Z})=T$. Next, let $U \subset T$ be the (integral) open subscheme such that fibers $\pi_{\mathcal Z}^{-1}(t) \subset X \times \{t\} \cong X$ have the expected codimension, that is equivalent to say that $t \in U$ if and only if $D_{r-k}(\varphi_t)$ has codimension $k$. Then the base change $p : {\mathcal Z} \times_T U = \pi_{\mathcal Z}^{-1}(U) = {\mathcal Z}_U \to U$ is proper with geometric fibers being the degeneracy loci $Z_t \subset X$ of expected codimension $k$. 

We claim that $p$ is also flat. To see this, let $\O_X(1)$ be a very ample line bundle on $X$. First observe that, since $L=\pi_1^*\O_X(1)$ is $T$-ample by \cite[Prop.~13.64]{gw1}, the restriction $L_{|{\mathcal Z}}$ is still $T$-ample by \cite[Rmk.~13.61(2)]{gw1}. Therefore, again by \cite[Prop.~13.64]{gw1}, the base change $L_U=(L_{|{\mathcal Z}})_{|\pi_{\mathcal Z}^{-1}(U)}$ of $L_{|{\mathcal Z}}$ is $U$-ample. On the other hand, $(L_U)_{|p^{-1}(t)}=\O_{Z_t}(1)$ for all $t \in U$ and all $p^{-1}(t)=Z_t \subset X \subset\P^N$ have the same Hilbert polynomial with respect to $\O_X(1)$: In fact, their structure sheaf, fits, by the Eagon-Northcott complex, into an exact sequence of the form 
\begin{multline*}
0 \to \O_{X}(-D)^{\oplus\binom{r-1}{r-k}} \to \E(-D)^{\oplus\binom{r-2}{r-k}} \to \Lambda^2\E(-D)^{\oplus\binom{r-3}{r-k}}\to\dots\to \Lambda^{k-1}\E(-D)\to\O_X\to \O_{Z_t} \to 0
\end{multline*}
where $\det \E = \O_X(D)$, 
hence
$\chi(\O_{Z_t}(m))$
is independent of $t$. Then the claim follows from \cite[Thm.~23.155]{gw2}. Since $p :  {\mathcal Z}_U \to U$ is proper and flat, the set of points $V \subset U$ such that $p^{-1}(t)=Z_t$ is reduced for $t \in V$ is open (see \cite[Thm.~12.2.1]{g} or \cite[(E.1)(11)]{gw1}). This proves the first part of the claim. Taking another base change, we get a proper flat morphism $q : {\mathcal Z}_U \times_U V = p^{-1}(V) \to V$ over an integral scheme such that every geometric fiber $q^{-1}(t) = D_{r-k}(\varphi_t) \subset X$ is reduced of codimension $k$. Then \cite[Lemma 37.53.8, \href{https://stacks.math.columbia.edu/tag/0E0N}{Tag 0E0N}]{st} tells that the number of connected components of fibers is constant. Since the degeneracy loci $D_{r-k}(\varphi)$ under consideration are reduced and have codimension $k$, they arise as fibers over some $t \in V$ and therefore they have the same number of connected components. In particular this holds for Ulrich subvarieties by their definition and Lemmas \ref{ulr}(i) and \ref{gg}.
\end{proof}

We now proceed with the proof of Theorem \ref{connesse}. We first study the case $c_{k+1}(\E)=0$, in which we can actually say a bit more.

\begin{lemma}
\label{connesse1}
Let $\E$ be a rank $r$ globally generated bundle on $X$ with $c_k(\E) \ne 0, c_{k+1}(\E)=0$ for some integer $k$ and let $s=h^0(\E^*)$. Consider morphisms $\varphi : \O_X^{\oplus (r+1-k)} \to \E$ such that the degeneracy loci $D_{r-k}(\varphi)$ are reduced of pure codimension $k$. If $k \in \{1, 2\},$ or if $k\ge3$ and $H^t((\Lambda^{k-1-t}\E)(-\det \E))=0$ for $1\le t \le k-2$, then all degeneracy loci $D_{r-k}(\varphi)$ as above have at least $r+1-k-s$ connected components.
\end{lemma}
\begin{proof}
Note that $k \le \min\{r,n\}$ since $c_k(\E) \ne 0$. Let $V_0 \subset H^0(\E)$ be a general subspace of dimension $r+1-k$ and let $\varphi_0 : V_0 \otimes \O_X \to \E$. If $r=k$, we set $\F=\E$. If $r \ge k+1$, let $V_1 \subset V_0$ be a general subspace of dimension $r-k$. Consider $\varphi_1 : V_1 \otimes \O_X  \to \E$, set $\F = \Coker \varphi_1, \H = \Coker \varphi_0$, so that we have an exact sequence
\begin{equation}
\label{e1}
0 \to V_0 \otimes \O_X \to \E \to \H \to 0
\end{equation}
and a commutative diagram
$$\xymatrix{& 0 \ar[d] & 0 \ar[d] & & \\ 0 \ar[r] & V_1 \otimes \O_X \ar[d] \ar[r]^{\hskip .4cm \varphi_1} & \E \ar[d] \ar[r] & \F  \ar[r] & 0 & \\ 0 \ar[r] & V_0 \otimes \O_X \ar[r]^{\hskip .4cm \varphi_0} \ar[d] & \E \ar[r] \ar[d] & \H \ar[r] & 0 &  \\ & \O_X \ar[r] \ar[d] & 0 \ar[d] & & \\ & 0 & 0 & & }$$
Since $c_{k+1}(\E)=0$, we deduce by Lemma \ref{gg} when $k \le n-1$ and by Lemma \ref{red} when $k=n$, that $D_{r-k-1}(\varphi_1)=\emptyset$, hence $\F$ is a rank $k$ vector bundle on $X$ and the snake lemma applied to the above diagram gives an exact sequence
\begin{equation}
\label{e2}
0 \to \O_X \to \F \to \H \to 0.
\end{equation}
Note that the above holds also when $r=k$. Let $W=D_{r-k}(\varphi_0)$, so that $W$ is reduced of pure codimension $k$ by Lemmas \ref{gg} and \ref{red}. When $r=k$, we have that $V_0 = \langle \sigma \rangle$ and $W=Z(\sigma)$. When $r \ge k+1$, let $\{\sigma_1, \ldots, \sigma_{r-k}\}$ be a basis of $V_1$ and let $\{\sigma, \sigma_1, \ldots, \sigma_{r-k}\}$ be a basis of $V_0$. We claim that 
\begin{equation}
\label{e3}
W=Z(\alpha(\sigma))
\end{equation}
where $\alpha=H^0(\psi) : H^0(\E) \to H^0(\F)$ is the map induced by the exact sequence
\begin{equation}
\label{e4}
0 \to V_1 \otimes \O_X \mapright{\varphi_1}  \E \mapright{\psi} \F \to 0.
\end{equation}
In fact, in a given point $x \in X$, we have the commutative diagram of $k(x)=\O_{X,x}/\mathfrak{m}_x$-vector spaces
$$\xymatrix{& 0 \ar[d] & 0 \ar[d] & & \\ 0 \ar[r] & V_1 \otimes \O_X(x) \ar[d] \ar[r]^{\hskip .4cm \varphi_1(x)} & \E(x) \ar[d] \ar[r]^{\hskip -.3cm \psi(x)}  & \F(x)  \ar[r] & 0 & \\ & V_0 \otimes \O_X(x) \ar[r]^{\hskip .4cm \varphi_0(x)} \ar[d] & \E(x) \ar[d] & & &  \\ & \O_X(x) \ar[r] \ar[d] & 0 & & \\ & 0 & & & }$$
where the first row is exact since $D_{r-k-1}(\varphi_1)=\emptyset$. Note that $\psi(x)(\sigma(x))=(\alpha(\sigma))(x)$. Now $x \in W$ if and only if $\rk \varphi_0(x) \le r-k$, if and only if $\Im \varphi_0(x)=\Im \varphi_1(x)$, if and only if $\sigma(x) \in \Ker \psi(x)$, if and only if $(\alpha(\sigma))(x)=0$, if and only if $x \in Z(\alpha(\sigma))$. Thus, \eqref{e3} is proved. Dualizing \eqref{e2}, we get the exact sequence
\begin{equation}
\label{en}
0 \to \H^* \to \F^* \to \O_X \to \SExt^1_{\O_X}(\H,\O_X) \to 0
\end{equation}
so that, by \eqref{e3}, $\SExt^1_{\O_X}(\H,\O_X) \cong \O_W$. Therefore, dualizing \eqref{e1}, we find the exact sequence
$$0 \to \H^* \to \E^* \to V_0^* \otimes \O_X \to  \O_W \to 0$$
that splits into the exact sequences
\begin{equation}
\label{e5}
0 \to \H^* \to \E^* \to  \G \to 0
\end{equation}
and
\begin{equation}
\label{e6}
0 \to \G \to V_0^* \otimes \O_X \to \O_W \to 0.
\end{equation}
Assume, for the time being, that 
\begin{equation}
\label{e7}
H^1(\H^*)=0. 
\end{equation}
Then we have that $h^0(\G) \le h^0(\E^*)=s$ by \eqref{e5} and \eqref{e7}. But then \eqref{e6} gives that
$$r+1-k-s \le h^0(V_0^* \otimes \O_X) - h^0(\G) \le h^0(\O_W)$$
and therefore $W$ has at least $r+1-k-s$ connected components. Hence the same holds for all degeneracy loci $D_{r-k}(\varphi)$ that are reduced of pure codimension $k$ by Proposition \ref{conn}. 

It remains to show \eqref{e7}, for which we distinguish in cases, according to $k$. Let $\O_X(D) = \det \E = \det \F$. If $k=1$, we have that $\F = \det \E$ and \eqref{e2} shows that $\H \cong (\det \E)_{|W}$. Hence $\H^*=0$ and \eqref{e7} holds in this case. If $k=2$, we have that $\dim V_0 = r-1$, so that $\H \cong \I_{W/X}(D)$ by the Eagon-Northcott resolution. Since $\SHom_{\O_X}(\I_{W/X}, \O_X) \cong \O_X$ by \cite[Lemma IV.5.1]{ak}, we have that $\H^* \cong \SHom_{\O_X}(\I_{Z/X}(D), \O_X) \cong \O_X(-D)$, and then \eqref{e7} follows from Lemma \ref{c3no}(ii). As for the case $k\ge3$, we will use the hypothesis $H^t((\Lambda^{k-t-1}\E)(-D))=0$ for $1\le t\le k-2$ to prove \eqref{e7}. To this end, we first prove the ensuing

\begin{claim}
\label{kos}
We have:
\begin{itemize}
\item[(i)] $H^{k-1}(\O_X(-D))=0$.
\item [(ii)] $H^i((\Lambda^{k-i-1}\F)(-D))=0$, for $1 \le i \le k-2$. 
\end{itemize} 
\end{claim}
\begin{proof}
Since $c_k(\E) \ne 0$, (i) follows from Lemma \ref{c3no}(ii). As for (ii), consider, for each $i \in \{1,\ldots, k-2\}$, the Eagon-Northcott-type exact sequence associated to \eqref{e4}:
$$0 \to F_{k-i-1} \to \cdots \to F_1 \to F_0 \to (\Lambda^{k-i-1}\F)(-D)\to 0$$
where $F_j = S^j V_1\otimes (\Lambda^{k-i-j-1}\E)(-D), 0 \le j \le k-i-1$. In order to prove (ii), it will suffice, by \cite[Prop.~B.1.2(i)]{la1}), to show that $H^{i+j}(F_j)=0$ for all $0 \le j\le k-i-1$. Now, if $0\le j\le k-i-2$, we have that
$H^{i+j}(F_j)=S^jV_1\otimes H^{i+j}((\Lambda^{k-i-j-1}\E)(-D))=0$ by assumption. If $j=k-i-1$ we have that
$H^{k-1}(F_{k-i-1})=S^{k-i-1}V_1\otimes H^{k-1}(\O_X(-D))=0$ by (i). This proves Claim \ref{kos}.
\end{proof}
We now continue the proof of the lemma. By \eqref{e3}, we have the Koszul resolution
$$0 \to \Lambda^k \F^*\to \Lambda^{k-1}\F^\ast\to\cdots \to \Lambda^2 \F^* \to \F^* \to \I_{W/X} \to 0$$
that, using \eqref{en} and the fact that $\SExt^1_{\O_X}(\H,\O_X) \cong \O_W$, can be split into 
\begin{equation}
\label{e99}
0 \to \O_X(-D) \to \Lambda^{k-1} \F^* \to \cdots \to \Lambda^2 \F^* \to \H^* \to 0.
\end{equation}
Again by \cite[Prop.~B.1.2(i)]{la1}, in order to prove \eqref{e7}, we will need to show that $H^i(\Lambda^{i+1}\F^*)=0$ for $1 \le i \le k-1$. Finally, the latter follows from Claim \ref{kos} since $H^i(\Lambda^{i+1}\F^*) \cong H^i((\Lambda^{k-i-1}\F)(-D))$. This concludes the proof of the lemma.
\end{proof}

Regarding the assumption $H^t((\Lambda^{k-t-1}\E)(-\det\E))=0$ for $1 \le t \le k-2$ in Lemma \ref{connesse1}, we can find many examples of Ulrich bundles satisfying it. 

\begin{lemma}
\label{exam} 
Let $X \subset \P^N$ be a smooth irreducible variety of dimension $n \ge 3$ and degree $d$. Let $\E$ be a rank $r$ Ulrich bundle on $X$. Let $k \in \Z$ be such that $3 \le k \le r$. Consider the following conditions:
\begin{itemize}
\item[(i)] $\det\E=\O_X(u)$.
\item[(ii)] $\det\E=\O_X(u)$ and $X \subset \P^N$ is subcanonical of degree $d \ge 3$.
\item[(iii)] $\det\E=\O_X(u), n \ge k-1$ and $X \subset \P^N$ is subcanonical with $K_X=\O_X(-i_X)$ such that $i_X \le n+3-k$ and $\O_X(1)$ is $2n$-Koszul (see \cite{to} for the definition of $M$-Koszul line bundle).
\end{itemize}
Then $H^t((\Lambda^{k-t-1}\E)(-\det\E))=0$ for all $1\le t\le k-2$ is implied by (i) if $k=3,$ by (ii) if $k=4,$ by (iii) for any $k \ge 5$.
\end{lemma}
\begin{proof} 
The condition for $k=3$ is just $H^1(\E(-u))=0$ and it is immediately satisfied because $\E$ is aCM by Lemma \ref{ulr}(vi). Now let $k \ge 4$ and suppose that $X \subset\P^N$ is subcanonical with $K_X=\O_X(-i_X)$. Then $H^{k-2}(\E(-u))=0$ because $\E$ is aCM by Lemma \ref{ulr}(vi) and $n \ge k-1$ in any case. Using Serre duality we get 
\begin{equation}
\label{eq11}
h^t(\Lambda^{k-t-1}\E(-u))=h^{t}(\Lambda^{r+t+1-k}\E^\ast)=h^{n-t}(\Lambda^{r+t+1-k}\E(-i_X))
\end{equation} 
for all $1\le t\le k-3$. The claim for $k=4$ follows from \cite[Lemma 4.2(iii)]{lr3} since $d \ge 3$, hence $i_X\le n-1$. For $k \ge 5$, consider the assumptions in (iii). Then $\Lambda^q\E$ is still $0$-regular for every $q \ge 1$ by \cite[Thm. 3.4]{to} as $\O_X(1)$ is $2n$-Koszul. It follows that $H^i((\Lambda^q \E)(l))=0$ for $i \ge 1, l \ge -i, q \ge 1$. As $-i_X \ge -n-3+k \ge -n+t$ for $1 \le t \le k-3$, we obtain the conclusion by \eqref{eq11}.
\end{proof}

We point out that, as shown in the proof of Theorem \ref{connesse} below, to obtain, for any globally generated bundle $\E$, the vanishing $H^1(\E(-\det\E))=0$, it is enough to suppose that $c_3(\E) \ne 0$ and $H^1(\O_X)=0$, which is a much weaker assumption than Lemma \ref{exam}(i). Anyway, triples $(X,\O_X(1),\E)$ satisfying the conditions in Lemma \ref{exam} exist. For (i) and (ii), the conditions hold for example if $\Pic(X) \cong \Z \O_X(1)$ and $k=3$, or $k=4$ and $d \ge 3$. As for (iii), some examples can be obtained in $({\mathbb G}(1,r+1),\O_G(1))$ by \cite[Thm.~7.2.5]{cmrpl} and \cite{ra}.

When $k=2$, we can say more.

\begin{lemma}
\label{numcc}
In the case $k=2$, further assume in Lemma \ref{connesse1} that $c_1(\E)^3 \ne 0.$ Let $h=h^1(\E^*)$. Then:
\begin{itemize}
\item[(a)] If $h=0$, then the connected components of such $D_{r-2}(\varphi)$'s are exactly $r-s-1$.
\item[(b)] If $h>0,s=0$ and $H^1(\O_X)=0$, then the connected components of such $D_{r-2}(\varphi)$'s are exactly $r+h-1$. Moreover there exists a globally generated vector bundle $\tilde\E$ of rank $r+h$ on $X$ fitting into a non-split exact sequence 
$$0 \to \O_X^{\oplus h} \to \tilde\E \to \E \to 0$$ 
with $H^0(\tilde\E^*)=H^1(\tilde\E^*)=0$ and such that all reduced degeneracy loci $D_{r+h-2}(\tilde\varphi)$ of pure codimension $2$ for injective morphisms $\tilde\varphi : \O_X^{\oplus(r+h-1)} \to \tilde\E$ have exactly $r+h-1$ connected components. In addition to this, for any fixed such $W'=D_{r-2}(\varphi')$, there exists an injective morphism $\tilde\varphi' : \O_X^{\oplus(r+h-1)}\to \tilde\E$ such that $D_{r+h-2}(\tilde\varphi')=W'$.
\item[(c)] If $r=2$ and $s=0,$ then $H^1(\E(-D))\cong H^1(\E^*)\cong H^1(\I_{D_{r-2}(\varphi)/X})$ for any such degeneracy loci $D_{r-2}(\varphi)$.
\end{itemize}
\end{lemma}
\begin{proof}
The assumption $c_1(\E)^3 \ne 0$ gives $H^j(\O_X(-D))=0$ for $0 \le j \le 2$ by Lemma \ref{c3no}(ii). Also we have that $\H^* \cong \O_X(-D)$. To see (a), since $H^1(\E^*)=0$, we get from \eqref{e5} that $h^0(\G)=s$ and $H^1(\G)=0$. Therefore, \eqref{e6} implies that $h^0(\O_W)=h^0(V_0^* \otimes\O_X) =r-s-1$, as claimed. 

As for (b), \eqref{e5} yields $h^0(\G)=0$ and $H^1(\G)\cong H^1(\E^\ast)$. Therefore, by \eqref{e6}, we obtain $h^0(\O_W)=r+h-1$ as desired. 
Now, take $\Xi= H^1(\E^*)$ and suppose that $h=dim(\Xi)>0$. Since 
$$\Id_{\Xi} \in Hom(\Xi, \Xi) \cong \Xi^* \otimes \Xi \cong \Ext^1(\E, \Xi^* \otimes \O_X)$$ 
this gives a non-split exact sequence of vector bundles 
\begin{equation}
\label{ti}
0 \to \Xi^* \otimes \O_X \mapright{\xi} \ \tilde\E \to \E \to 0
\end{equation}
with $\tilde\E$ of rank $r+h$. It follows that $c_i(\tilde\E)=c_i(\E)$ for all $i\ge1$ and that $\tilde\E$ is globally generated since $\E$ is and $H^1(\O_X)=0$. Dualizing \eqref{ti} and taking cohomology we get, using $H^0(\E^*)=0$, an exact sequence
$$0 \to H^0(\tilde\E^*) \to \Xi \mapright{\delta} \Xi \to H^1(\tilde\E^*) \to 0$$
where, by construction, $\delta=\Id_{\Xi}$. It follows that $h^0(\tilde\E^*)=h^1(\tilde\E^*)=0$. Now, the map on global sections $H^0(\tilde\E) \mapright{\alpha} H^0(\E)$ is surjective as $H^1(\O_X)=0$. Let $V' \subset H^0(\E)$ with $\dim V' = r-1$ and such that $W':=D_{r-2}(\varphi')$ is reduced of pure codimension $2$, where $\varphi' : V'\otimes\O_X \to \E$. Set $\tilde V'= \alpha^{-1}(V') \subset H^0(\tilde\E)$. Then $\tilde V' \cong \Xi^* \oplus V'$ and there is a commutative diagram
$$\xymatrix{0 \ar[r] &  \Xi^*\otimes\O_X \ar[d]^{\Id_{\Xi^*}} \ar[r] & \tilde V'\otimes\O_X \ar[d]^{\tilde\varphi'} \ar[r] & V'\otimes\O_X  \ar[r] \ar[d]^{\varphi'} & 0 & \\ 0 \ar[r] & \Xi^*\otimes\O_X \ar[r]^\xi & \tilde\E \ar[r] & \E \ar[r] & 0. &}$$
To show that $D_{r+h-2}(\tilde\varphi')=W'$, observe that, since $\xi$ never drops rank, the second row remains exact after tensoring with $k(x)=\O_{X,x}/\mathfrak{m}_x$ for any point $x \in X$. Hence we have $\tilde\E(x)\cong \Xi^\ast(x)\oplus\E(x),$ where $\Xi^\ast(x)\cong (\Xi^*\otimes\O_X)(x)$ is the isomorphic image of $\xi(x)$. Therefore $\tilde\varphi'(x)=\xi(x) \oplus \varphi'(x)$. Since $\xi(x)$ has maximal rank, this says that $\tilde\varphi'(x)$ drops rank if and only if $\varphi'(x)$ does, which means that $D_{r+h-2}(\tilde\varphi')$ and $D_{r-2}(\varphi')$ coincide as subschemes of $X$, because they are locally defined by the vanishing of the same minors (see Remark \ref{degeneracy}). By part (a) we get (b). Finally, to prove (c), observe that \eqref{e1} twisted by $\O_X(-D)$ reads
$$0 \to \O_X(-D) \to \E^* \to \I_{W/X} \to 0$$
where we used $\E^*\cong \E(-D)$. Taking cohomology and recalling that $H^j(\O_X(-D))=0$ for $1 \le j \le 2$, we get the claim.
\end{proof}

For $k=3$ we have.

\begin{lemma}
\label{sharp}
In the case $k=3$, further assume in Lemma \ref{connesse1} that $X \subset \P^N$ is subcanonical with $n \ge 4$ and that $\E$ is aCM with $\det\E=\O_X(u), u>0$. Then the connected components of such $D_{r-3}(\varphi)$'s are exactly $r-s-2$.
\end{lemma}
\begin{proof} 
Since $\E$ is aCM and $X$ is subcanonical, we have that $H^i(\E(-u))=0$ for $1 \le i \le 2$ and $H^1(\E^*) \cong H^{n-1}(\E(-i_X))=0$. Now, $c_3(\E) \ne 0$, hence also $c_2(\E) \ne 0$ by Lemma \ref{gg}. Choosing a general subspace $V \subset H^0(\E)$ of dimension $r-1$, we get that \eqref{succ} holds with $Z \ne \emptyset$ by Lemma \ref{gg}. Hence \eqref{succ} implies that $H^0(\E(-u))=0$. Also, $H^i(\O_X(-u))=0$ for $1 \le i \le 3$ by Kodaira vanishing. Using the exact sequence \eqref{e4} we deduce that $H^i(\F(-u))=0$ for $0 \le i \le 2$. Since $\F$ has rank $3$, \eqref{e99} becomes
$$0 \to \O_X(-u) \to \F(-u) \to \H^* \to 0$$
and we find that $H^i(\H^*)=0$ for $0 \le i \le 2$. Now, \eqref{e5} gives that $H^1(\G)=0$ and $h^0(\G)=s$ and then \eqref{e6} implies that $h^0(\O_W)=r-s-2$ as required. 
\end{proof}

We now prove Theorem \ref{connesse}.

\renewcommand{\proofname}{Proof of Theorem \ref{connesse}}
\begin{proof}
Let $H$ be a very ample divisor on $X$. We have that $r \ge k+s$ by Lemma \ref{c3no}(iii) and $k \le n$ since $c_k(\E) \ne 0$. 

First, assume that $k \in \{1, 2\}$. The fact that (i) implies (ii) is the content of Lemma \ref{connesse1}. Now assume (iii). If $k=n$ we have obviously $c_{k+1}(\E)=0$, hence we can assume that $k \le n-1$. Let $V_0 \subset H^0(\E)$ be a general subspace of dimension $r+1-k$ and let $\varphi_0 : V_0 \otimes \O_X \to \E$. It follows from Lemmas \ref{gg} and \ref{red} that $Z:=D_{r-k}(\varphi_0)$ is reduced of pure codimension $k$ and $Z$ is not connected by hypothesis. 

If $n=k+1$, we set $X'=X, \E'=\E$ and $Z'=Z$. If $n \ge k+2$, cutting down with $n-k-1$ general $H_1, \ldots, H_{n-k-1} \in |H|$, we get a smooth $(k+1)$-fold $X'$ and a globally generated bundle $\E'=\E_{|X'}$ with $c_k(\E') \ne 0$ by Lemma \ref{gg}(ii). Moreover, observe that
$$Z \cap H_1 \cap \ldots \cap H_{n-k-1} = Z \cap X'=D_{r-k}(\varphi_0) \cap X' =D_{r-k}({\varphi_0}_{|X'})$$
is reduced of pure codimension $k$ and disconnected. Let $V' \subset H^0(\E')$ be a general subspace of dimension $r+1-k$, let $\varphi' : V' \otimes \O_{X'} \to \E'$, so that $Z':=D_{r-k}(\varphi')$ is reduced of pure codimension $k$ by Lemmas \ref{gg} and \ref{red}. Also, $Z'$ is disconnected by Proposition \ref{conn}. 

If $k=1$, we have that $X'$ is a smooth surface and $Z' \in |\det \E'|$ is disconnected, hence, since $\det \E'$ is globally generated, it follows from Bertini's theorem that $c_1(\E')^2=0$, and therefore $c_2(\E)=0$ by Lemmas \ref{c3no}(i) and \ref{gg}(ii).

If $k=2$, Lemma \ref{zeta1} gives a smooth irreducible surface $Y' \in |\det \E'|$ containing $Z'$, with $\O_{Y'}(Z')$ globally generated. Since $Z' \in |\O_{Y'}(Z')|$ is disconnected, Bertini's theorem implies that $|\O_{Y'}(Z')|$ is composite with a pencil, and therefore $(Z')^2=0$. Then, \eqref{c33} gives that $c_3(\E)H^{n-3}= c_3(\E')=(Z')^2=0$, hence $c_3(\E)=0$ by Lemma \ref{gg}(ii). 

Thus, (i) is proved in both cases $k \in \{1, 2\}$. Now, if $r \ge k+s+1$, then clearly (ii) implies (iii) and therefore, if $r \ge k+s+1$, we have that (i), (ii) and (iii) are equivalent.

Next, we show that (i) implies (ii) when $k=3$ and $H^1(\O_X)=0$. To this end, setting $\O_X(D)= \det \E$,  it is enough to prove by Lemma \ref{connesse1} that 
\begin{equation}
\label{e-d}
H^1(\E(-D))=0.
\end{equation} 
Let $V \subset H^0(\E)$ be a general subspace of dimension $r-1$ and let $Z=D_{r-2}(\varphi)$, where $\varphi : V \otimes \O_X \to \E$. Since $c_3(\E) \ne 0$ by hypothesis, it follows from Lemma \ref{gg} that $c_2(\E) \ne 0$. Hence $Z$ is connected, since for $k=2$ we have already proved that (iii) implies (i). Also, Lemma \ref{c3no}(ii) gives that $H^1(\O_X(-D))=0$. Since $H^1(\O_X)=0$, the exact sequence
$$0 \to \I_{Z/X} \to \O_X \to \O_Z \to 0$$
shows that $H^1(\I_{Z/X})=0$. Then \eqref{e-d}  follows from the exact sequence
$$0 \to V \otimes \O_X(-D) \to \E(-D) \to \I_{Z/X} \to 0.$$
Now, assume again that $k \in \{1, 2\}$. To see (iv), observe that if $D_{r-k}(\varphi)$ is disconnected, then (iii) holds and therefore so does (i), contradicting Proposition \ref{b+}(i). As for (v), assume that $\varphi : \O_X^{\oplus (r+1-k)} \to \E$ is a general morphism and that $D_{r-k}(\varphi)$ is singular. It follows from Lemma \ref{red}, that $D_{r-k-1}(\varphi)=\Sing(D_{r-k}(\varphi)) \ne \emptyset$ of the expected codimension $2k+2$. Then $[D_{r-k-1}(\varphi)]=c_{k+1}(\E)^2-c_k(\E)c_{k+2}(\E)$ by Porteous' formula (see for example \cite[Thm.~12.4]{eh}). Hence, if $D_{r-k}(\varphi)$ were disconnected, we would have, by Theorem \ref{connesse}, that $c_{k+1}(\E)=0$ and also $c_{k+2}(\E)=0$ by Lemma \ref{gg}, hence the contradiction $[D_{r-k-1}(\varphi)]=0$.
\end{proof}
\renewcommand{\proofname}{Proof}

The following simple example shows that Theorem \ref{connesse} (or any possible generalization to $k \ge 3$) is sharp regarding the inequality $r \ge s+k+1$.

\begin{remark}
\label{bana}
Let $1 \le k \le \min\{r,n\}$, let $\G=\O_{\P^k}(1)^{\oplus k} \oplus \O_{\P^k}^{\oplus (r-k)}$, let $X=\P^{n-k} \times \P^k$ and let $\E=\pi^*\G$, where $\pi : \P^{n-k} \times \P^k \to \P^k$ is the second projection. Then we have that $\E$ is globally generated, $s=h^0(\E^*)=r-k$ and, for any $t \in \P^k$, $c_k(\E)=[\P^{n-k} \times \{t\}] \ne 0, c_{k+1}(\E)=0$. Moreover, if $\varphi: \O_X^{\oplus (r+1-k)} \to \E$ is general, then $D_{r-k}(\E)=\P^{n-k} \times \{t\}$ is connected. 
\end{remark}

Next, we prove Corollary \ref{singconn}.

\renewcommand{\proofname}{Proof of Corollary \ref{singconn}}
\begin{proof}
We have that $\E$ is globally generated by Lemma \ref{ulr}(i) and $h^0(\E^*)=0$, for otherwise Lemma \ref{ulr}(iv) gives the contradiction $c_2(\E)=0$. Hence Theorem \ref{connesse} with $k=2$ applies to Ulrich subvarieties by their same definition and Lemma \ref{gg}. Thus, we get (v) and the equivalence of (i), (ii) and (iii). As for (iv), by \cite[Thm.~2]{bu} we know that 
\begin{equation}
\label{cop}
\B_+(\E)=\bigcup_L L
\end{equation}
where $L$ ranges over all lines $L \subset X$ such that $\E_{|L}$ is not ample. The assumption in (iv) and \eqref{cop} say exactly that $\B_+(\E) \ne X$, therefore (iv) follows from Theorem \ref{connesse}(iv).
\end{proof}
\renewcommand{\proofname}{Proof}

\begin{remark}
\label{b+ulrich}
Let $n \ge 3, r \ge 3$, and let $\E$ be a rank $r$ Ulrich bundle on $X \subset \P^N$. If $X$ is not covered by lines, then all Ulrich subvarieties associated to $\E$ are connected by Corollary \ref{singconn}(iv). However, Ulrich subvarieties can be connected even if $X$ is covered by lines. For instance, there are varieties $X\subset\P^N$ covered by lines supporting very ample Ulrich bundles $\E$ (see \cite[Rmk.~4.3(i)]{ls}). Then $\E_{|L}$ is very ample on all lines $L\subset X$, whence Corollary \ref{singconn}(iv) applies. 
\end{remark}

\begin{remark}
The converse of  Theorem \ref{connesse}(iv) does not hold. For instance, as we will see in Proposition \ref{non big}(i), all non-big Ulrich bundles $\E$ on threefolds $X \subset \P^N$, thus having $\B_+(\E)=X$, which satisfy $c_1(\E)^3>0$, have connected associated Ulrich subvarieties. Other examples are quadrics $Q_n \subset \P^{n+1}$ for $n\ge 3$ with $\E$ being any Ulrich bundle: Indeed all Ulrich subvarieties associated to $\E$ are connected by Lemma \ref{h1=0}(ii), but $\B_+(\E)=Q_n$ by \cite[Cor.~1.6]{o} and \cite[Thm.~2]{bu}, because all Ulrich bundles are direct sum of spinor bundles.
\end{remark}

\begin{remark}
\label{numcc2}
If $\E$ is Ulrich with $c_2(\E) \ne 0, c_3(\E)=0$ and $X \subset \P^N$ is subcanonical of dimension $n \ge 4$, then all Ulrich subvarieties associated to $\E$ have exactly $r-1$ connected components, unless $(X, \O_X(1), \E)=(\P^2 \times \P^2, \O_{\P^2}(1) \boxtimes \O_{\P^2}(1), \pi^*(\O_{\P^2}(2))^{\oplus r})$, where $\pi : \P^2 \times \P^2 \to \P^2$ is a projection, and, in the latter case, all Ulrich subvarieties have exactly $2r(r-1)$ connected components. Indeed, observe that $s=h^0(\E^*)=0$ by Lemma \ref{ulr}(iv) and $h^1(\E^*)=h^{n-1}(\E(K_X))=0$ by Lemma \ref{ulr}(vi). Hence, if $c_1(\E)^3 \ne 0$, Lemma \ref{numcc} applies. If $c_1(\E)^3=0$, we can repeat the proof of \cite[Cor.~4]{ls} using now the fact that $c_1(\E)^3=0$ and that $3 \le \lfloor \frac{n}{2}+1 \rfloor$. It follows from that proof that $(X,\O_X(1),\E)$ is a linear Ulrich triple over a surface, because $c_2(\E) \ne 0$. On the other hand, $X$ is subcanonical, hence the only possibility is that $K_X=-(n-1)H$, so that $X$ is a del Pezzo manifold. Since $\rho(X) \ge 2$, by the classification of del Pezzo manifolds (see for example \cite[pages 860-861]{lp}, \cite[Table, page 710]{f1}), we deduce the only possible case $(X, \O_X(1))=(\P^2 \times \P^2, \O_{\P^2}(1) \boxtimes \O_{\P^2}(1))$. Then, it follows from \cite[Cor.~4.9]{ls} that $\E=\pi^*(\O_{\P^2}(2))^{\oplus r})$. Therefore Lemma \ref{usub} implies that general Ulrich subvarieties associated to $\E$ are a disjoint union of $2r(r-1)=c_2(\O_{\P^2}(2)^{\oplus r})$ planes. Hence all Ulrich subvarieties have exactly $2r(r-1)$ connected components by Proposition \ref{conn}.
\end{remark}

\section{Some connectedness statements and examples}

A special case in which degeneracy loci $D_{r-n}(\varphi)$ associated to $\E$ are connected is when $c_n(\E)=1$. We observe a few things about this.

\begin{remark} (We thank F. Moretti for observing that we could use \cite[Prop.~1.5]{mo}.) 
\label{moretti}
Let $\E$ be a rank $r$ globally generated bundle on $X$ with $c_n(\E)=1$. Then $X$ is rational. Indeed, since $c_n(\E)=1$, we have that $r \ge n$. Now, if $r=n$, set $\F=\E$. If $r > n$, consider a general morphism $\varphi : \O_X^{\oplus (r-n)} \to \E$. It follows from Lemma \ref{red} that $D_{r-n-1}(\varphi)= \emptyset$, hence we have an exact sequence
$$0 \to \O_X^{\oplus (r-n)} \to \E \to \F \to 0$$
where $\F$ is a rank $n$ globally generated bundle with $c_n(\F)=c_n(\E)=1$. Now, in any case, $H^0(\F^*)=0$ by Lemma \ref{c3no}(ii). Note that $h^0(\F) \ge n+1$: if $h^0(\F)=n$, then $\F \cong \O_X^{\oplus n}$ and therefore $c_n(\F)=0$, a contradiction. Let $W \subseteq H^0(\F)$ be a general subspace of dimension $n+1$, so that $W$ generates $\F$ in codimension $2$, that is away from $D_{n-1}(\varphi_W)$, where $\varphi_W : W \otimes \O_X \to \F$, by Lemma \ref{red}. Moreover, if $\sigma \in W$ is a general section, that is a general section in $H^0(\F)$, we have that $Z(\sigma)$ is $0$-dimensional of degree $c_n(\F)=1$. Therefore \cite[Prop.~1.5]{mo} gives that $X$ is rational.
\end{remark}

In the case of Ulrich bundles, while the case $n=2$ is treated in Theorem \ref{superficie}, here we give a few examples in rank $n$ and we prove that, in many cases, $c_n(\E)=1$ cannot happen in rank $r > n$. 

\begin{remark}
\label{cn=1-bis}
Let $X \subset \P^{2n}$ be the rational normal scroll $\varphi_{\xi}(\P(\F))$, where $\F=\O_{\P^1}(1)^{\oplus (n-1)} \oplus \O_{\P^1}(2)$ and $\xi$ is the tautological line bundle. We have that $\L = \xi-\pi^*\O_{\P^1}(1)$ is Ulrich on $X$ (see Example \ref{ese}) and therefore so is $\E = \L^{\oplus n}$ and $c_n(\E)=\L^n=1$. Case (iii) of Theorem \ref{superficie} corresponds to $n=2$. 
\end{remark}

On quadrics we have

\begin{lemma}
\label{cn=1-tris}
Let $n \ge 2$ and let $Q_n \subset \P^{n+1}$ be a smooth quadric. The only rank $n$ Ulrich bundles $\E$ on $Q_n$ with $c_n(\E)=1$ are the following: 
\begin{itemize}
\item[(i)] $n=2, \E=\mathcal S' \oplus \mathcal S''$.
\item [(ii)] $n=4, \E=(\mathcal S')^{\oplus 2}$ or $(\mathcal S'')^{\oplus 2}$.
\item[(iii)] $n=8, \E=\mathcal S'$ or $\mathcal S''$.
\end{itemize} 
\end{lemma}
\begin{proof}
Since every Ulrich bundle on $Q_n$ is direct sum of spinor bundles (see for example \cite[Rmk.~2.5(4)]{bgs}), that have rank $2^{\lfloor \frac{n-1}{2} \rfloor}$, we have that $n$ is a multiple of $2^{\lfloor \frac{n-1}{2} \rfloor}$, and it follows easily that $n \in \{2, 4, 8\}$. Using \cite[Rmk.~2.9]{o}, we obtain that the only cases are the ones listed. Case (i) of Theorem \ref{superficie} corresponds to $n=2$. 
\end{proof}

\begin{remark}
\label{cn=1}
Assume that $X \subset \P^N$ is not covered by lines. Let $\E$ be a $0$-regular (in particular Ulrich) rank $r$ vector bundle with $c_n(\E)=1$. Then $r=n$. Indeed, $n \le r$ since $c_n(\E) \ne 0$ and it follows from \cite[Rmk.~7.3]{lo}, that $c_n(\E) \ge \binom{r}{n}>1$ if $r > n$. 
\end{remark}

In the next two lemmas we prove some connectedness statements. We set $X \subset \P^N$ to be a smooth irreducible variety of dimension $n$, $\E$ a rank $r$ globally generated bundle on $X$ with $c_2(\E) \ne 0$, $\det \E = \O_X(D)$ and $Z$ a normal pure codimension $2$ degeneracy locus $D_{r-2}(\varphi)$, where $\varphi: \O_X^{\oplus (r-1)} \to \E$ is an injective morphism.

\begin{lemma}
\label{q-a}
If $\E$ is $(n-3)$-ample (equivalently, when $\E$ is Ulrich, if either $X \subset \P^N$ does not contain a linear space of dimension $n-2$, or $\E_{|M}$ does not have a trivial direct summand for every linear space $M \subset X$ of dimension $n-2$), then $Z$ is connected. 
\end{lemma}
\begin{proof}
The equivalence is the content of \cite[Thm.~1]{lr2}. Connectedness follows from \cite[Thm.~6.4(a)]{t}. 
\end{proof}

\begin{lemma}
\label{h1=0}
Assume that $n \ge 3$. Then $Z$ is connected if one of the following holds:
\begin{itemize}
\item [(i)] $c_1(\E)^3 \ne 0$ and $H^1(\E(-D))=0$.
\item [(ii)] $\E$ is Ulrich and $\Pic(X) \cong \Z A$, with $A$ ample such that $h^0(A) \ge 2$. 
\item [(iii)] $\E$ is Ulrich, $D-K_X-(n+1)H$ is ample and $r \le n-1$, where $H$ is very ample (for example when $\E$ is special and $4 \le r \le n-1, r$ even).
\end{itemize}
Moreover, if $H^1(\O_X)=0$ and $Z$ is connected, then $H^1(\E(-D))=0$.
\end{lemma}
\begin{proof}
If (i) holds, we have that $H^2(\O_X(-D))=0$ by Lemma \ref{c3no}(ii), hence $H^1(\I_{Z/X})=0$ by the Eagon-Northcott resolution
\begin{equation}
\label{ea}
0 \to \O_X(-D)^{\oplus (r-1)} \to \E(-D) \to \I_{Z/X} \to 0.
\end{equation}
Therefore, the exact sequence
\begin{equation}
\label{ze}
0 \to \I_{Z/X} \to \O_X \to \O_Z \to 0
\end{equation}
shows that $h^0(\O_Z)=h^0(\O_X)=1$ and $Z$ is connected. To see (ii), note that $(X,\O_X(1)) \ne (\P^n,\O_{\P^n}(1))$, for otherwise $\E=\O_{\P^n}^{\oplus r}$ by Lemma \ref{ulr}(viii) and then $c_2(\E)=0$, a contradiction. Setting $K_X=-i_XA$, we deduce, as is well-known, that $i_X \le n$. Next, we can write $D=aA$ with $a>0$ by Lemma \ref{ulr}(iii). Hence $c_1(\E)^n>0$. Let $H=hA$, so that, if $h=1$, we have that $H^1(\E(-D))=H^1(\E(-a))=0$ by Lemma \ref{ulr}(vi). Now assume that $h \ge 2$, so that 
\begin{equation}
\label{ind}
\left(n+1-\frac{4}{r}\right)h \ge i_X
\end{equation}
holds, since $i_X \le n$ and $r \ge 2$, because $c_2(\E) \ne 0$. It follows from Lemma \ref{ulr}(x) that $a=\frac{r}{2}((n+1)h-i_X)$ and \eqref{ind} implies that $a=\frac{r}{2}((n+1)h-i_X) \ge 2h$. But then $H^1(\E(-D))=0$ by Lemma \ref{van}(ii). Therefore (i) holds and $Z$ is connected by (i). If (iii) holds, we have that $D=K_X+(n+1)H+A$ for an ample $A$, hence, as $K_X+(n+1)H$ is nef, $D^n>0$. Let $\F=\E^*(K_X+(n+1)H)$ be the dual Ulrich bundle, as in Lemma \ref{ulr}(ix). We have, by Serre's duality
$$h^1(\E(-D))=h^{n-1}(\omega_X \otimes \E^*(D))=h^{n-1}(\omega_X \otimes \F(D-K_X-(n+1)H))=0$$
by \cite[Ex.~7.3.17]{la2}. Thus, (iii) follows from (i). Finally, if $H^1(\O_X)=0$ and $Z$ is connected, \eqref{ze} shows that $H^1(\I_{Z/X})=0$. Since $c_2(\E) \ne 0$, we have that $H^1(\O_X(-D))=0$ by Lemma \ref{c3no}(ii) and \eqref{ea} gives $H^1(\E(-D))=0$.
\end{proof}
 
In the following standard examples we also have connectedness.

\begin{lemma}
\label{exa}
Let $B$ be a smooth irreducible curve and let $\F$ be a very ample rank $n \ge 3$ vector bundle on $B$. Let $\pi: X = \P(\F) \to B$ and $H \in |\O_{\P(\F)}(1)|$. Let $M$ be a line bundle on $B$ and let $\G$ be a vector bundle on $B$ such that $H^i(M)=H^i(\G)=0$ for every $i \ge 0$. Let $\E$ be a rank $r$ vector bundle that is an extension of type
$$0 \to \Omega_{X/B}(2H+\pi^*M) \to \E \to \pi^*(\G(\det \F)) \to 0.$$
Then $\E$ is an Ulrich bundle and $H^1(\E(-D))=0$, where $\det \E = \O_X(D)$. In particular any Ulrich subvariety associated to $\E$ is connected.
\end{lemma}
\begin{proof} 
Note that $\Omega_{X/B}(2H+\pi^*M)$ is Ulrich by \cite[Lemma 4.1]{lm} and $\pi^*(\G(\det \F))$ is Ulrich by \cite[Lemma 4.1]{lo}. Therefore also $\E$ is Ulrich. Also, $c_1(\E)^n>0$ by \cite[Lemma 4.1]{lm} and Lemma \ref{ulr}(iii). Now, the last assertion follows from the first and Lemma \ref{h1=0}(i). Next, we prove that $H^1(\E(-D))=0$. Note that $r \ge n-1$. We have that $D=(n-2)H+\pi^*N, \ \hbox{where} \ N=(n-1)M+\det \G +(r-n+2) \det \F$. Now
$$H^1((\pi^*(\G(\det \F)))(-D))=H^1((\pi^*(\G(\det \F-N)))(2-n)H))=0$$
since $R^j \pi_*(\O_{\P(\F)}(2-n))=0$ for every $j \ge 0$. Also, we have
$$\Omega_{X/B}(2H+\pi^*M)(-D))=\Omega_{X/B}((4-n)H+\pi^*(M-N))$$
and it remains to show that
\begin{equation}
\label{h1}
H^1(\Omega_{X/B}((4-n)H+\pi^*(M-N)))=0.
\end{equation}
If $n=3$ we have $R^j \pi_*(\Omega_{X/B}(H+\pi^*(M-N)))=0$ for every $j \ge 0$ by \cite[Lemma 7.3.11(i), (iii)]{la2}, hence \eqref{h1} follows. If $n \ge 4$ we use the exact sequence
\begin{equation}
\label{suc}
0 \to \Omega_{X/B}((4-n)H+\pi^*(M-N)) \to (\pi^* \F)((3-n)H+\pi^*(M-N)) \to (4-n)H+\pi^*(M-N) \to 0.
\end{equation}
Since $R^j \pi_*((\pi^* \F)((3-n)H+\pi^*(M-N)))=0$ for every $j \ge 0$, we get that 
$$H^1((\pi^* \F)((3-n)H+\pi^*(M-N)))=0.$$ 
Hence, using \eqref{suc}, to see \eqref{h1} and conclude the proof, it remains to show that 
\begin{equation}
\label{h0}
H^0((4-n)H+\pi^*(M-N))=0.
\end{equation}
If $n \ge 5$ we have that $H^0(-H+\pi^*(M-N))=H^0((\pi_*(-H))(M-N))=0$, that is \eqref{h0} holds.
If $n=4$, to see that \eqref{h0}, we will show that $\deg(M-N)<0$. Now
$$M-N=-2M-\det(\G)-(r-2)\det \F.$$
If $g$ is the genus of $B$, we have from $H^i(\G)=0$ for every $i \ge 0$ that $\deg \G=(r-3)(g-1)$, hence
$$\deg(M-N)=-(r-1)(g-1)-(r-2)\deg \F.$$
Since $r \ge 3$ and $\deg \F>0$, we see that $\deg(M-N)<0$ if $g \ge 1$. On the other hand, if $g=0$ we have, as $\F$ is very ample rank $4$ on $\P^1$, that $\deg \F \ge 4$, hence $\deg(M-N)=r-1-(r-2)\deg \F \le -3r+7<0$. This proves \eqref{h0} and the lemma.
\end{proof}

On the other hand, many times, Ulrich subvarieties will be disconnected. 
\begin{remark}
\label{718}
Note that, when $c_2(\E) \ne 0, c_1(\E)^3=0$ and $n \ge 4$, then all Ulrich subvarieties are disconnected. In fact, we can repeat the proof of \cite[Cor.~4]{ls} using now the fact that $c_1(\E)^3=0$ and that $3 \le \lfloor \frac{n}{2}+1 \rfloor$. It follows from that proof that $(X,\O_X(1),\E)$ is a linear Ulrich triple over a surface, because $c_2(\E) \ne 0$. Then Lemma \ref{lut2} below shows that all Ulrich subvarieties are disconnected. 
\end{remark}

\begin{lemma}
\label{lut2}
Let $(X, \O_X(1), \E)$ be a linear Ulrich triple of dimension $n \ge 3$ over a smooth surface $B$.  Then the Ulrich subvarieties associated to $\E$ are connected if and only if $(X, \O_X(1), \E) \cong (\P^1\times\P^2, \O_{\P^1}(1) \boxtimes \O_{\P^2}(1), q^*(\O_{\P^2}(1)^{\oplus 2})),$ where $q : \P^1\times\P^2 \to \P^2$ is the second projection.
\end{lemma}
\begin{proof}
Consider the case $(X, \O_X(1), \E) \cong (\P^1 \times \P^2, \O_{\P^1}(1) \boxtimes \O_{\P^2}(1), q^*(\O_{\P^2}(1)^{\oplus 2}))$, that is a linear Ulrich triple over $\P^2$ by \cite[Lemma 4.1]{lo} and $c_2(q^*(\O_{\P^2}(1)^{\oplus 2}))=q^*H_{\P^2}^2 \ne 0$. Hence, for a general Ulrich subvariety $Z$ associated to $\E=q^*(\O_{\P^2}(1)^{\oplus 2})$, we have that $Z \ne \emptyset$ by Proposition \ref{c2=0}. Therefore, Lemma \ref{usub} gives that $\O_Z \cong \pi^*\O_{Z_{\P^2}}$, where $Z_{\P^2}=D_{r-2}(\varphi_{\P^2})$ is $0$-dimensional. But then $[Z_{\P^2}]=c_2(\O_{\P^2}(1)^{\oplus 2})=H_{\P^2}^2$ and we get that $Z_{\P^2}=\{P\}$ for some point $P \in \P^2$, hence $Z$ is connected and so are all Ulrich subvarieties by Proposition \ref{conn}.

Vice versa, suppose that $(X, \O_X(1), \E)$ is a linear Ulrich triple of dimension $n \ge 3$ over a surface $B$ such that all Ulrich subvarieties associated to $\E=\pi^*(\G(\det \F))$ are connected, where $\pi : X \cong \P(\F) \to B$. We first prove that $(X, \O_X(1), \E)$ is not a linear Ulrich triple over a smooth curve $B_1$. In fact, assume that $(X, \O_X(1), \E) \cong (\P(\F_1), \O_{\P(\F_1)}(1), \pi_1^*(\G_1(\det \F_1))$, where $\pi_1 : X \cong \P(\F_1) \to B_1$. Now, for any fiber $F_1$ of $\pi_1$, we have a morphism $\pi_{|F_1} : \P^{n-1} \cong F_1 \to B$ that is not constant, since the fibers of $\pi$ are $(n-2)$-dimensional. Hence $\pi_{|F_1}$ is finite-to-one onto its image and we get that $n-1 \le \dim B=2$, so that $n=3$. Hence the fibers $F$ of $\pi$ are lines and the fibers $F_1$ of $\pi_1$ are planes in $\P^N=\P H^0(\O_X(1))$. Moreover it cannot be that $F \subset F_1$, for otherwise $\pi_{|F_1}$ would contract $F$. Therefore $F \cap F_1$ is a point and $\pi_{|F_1}$ is an isomorphism. Hence $B \cong \P^2$ and therefore $h^1(\O_{B_1})=h^1(\O_X)=h^1(\O_{\P^2})=0$, so that $B_1 \cong \P^1$. Then it follows from \cite[Thm.~A and Rmk.~1.6]{sa} that $(X, \O_X(1)) \cong (\P^1\times\P^2, \O_{\P^1}(1) \boxtimes \O_{\P^2}(1))$ and $\pi_1$ is the first projection, $\pi$ the second. Thus, we get that $\F \cong \O_{\P^2}(1)^{\oplus 2}$ and $\F_1 \cong \O_{\P^1}(1)^{\oplus 3}$. Then \cite[Lemma 4.1]{lo} implies that $\G \cong \O_{\P^2}(-1)^{\oplus r}$ and $\G_1 \cong \O_{\P^1}(-1)^{\oplus r}$. But then $\pi_1^*(\O_{\P^1}(2)^{\oplus r}) \cong \E \cong \pi^*(\O_{\P^2}(1)^{\oplus r})$, giving the contradiction
$$0=(\pi_1^*(\O_{\P^1}(2r))^2 \cong c_1(\E)^2 = \pi^*(\O_{\P^2}(r))^2=r^2.$$
This proves that $(X, \O_X(1), \E)$ is not a linear Ulrich triple over a smooth curve.

By Proposition \ref{c2=0} we get that $c_2(\E) \ne 0$ and, since $c_3(\E)=\pi^*c_3((\G(\det\F)))=0$, we find by Corollary \ref{singconn} that $\E$ and then $\G$ have rank two. Let $Z$ be an Ulrich subvariety arising from a general section of $\E$. It follows from Lemma \ref{usub} that $\O_Z \cong \pi^*\O_{Z_B}$, where $Z_B$ is the zero locus of a general section of $\G(\det\F)$. Since $\G(\det\F)$ is globally generated (see for example \cite[Exc.~5.1.29(b)]{liu}) we have that $Z_B$ is smooth and connected (since $Z$ is) and therefore $Z_B=\{P\}$ for some point $P \in B$. By Lemma \ref{zeta1}, there is a smooth irreducible curve $C \in |\det(\G(\det\F))|$ such that $P \in C$ and $\O_C(P)$ is globally generated, so that $C \cong \P^1$. Moreover, $H^1(\O_B(-C)) \cong H^1(\O_X(-\det \E))=0$ by Lemma \ref{c3no}(ii), hence the exact sequence
$$0 \to \O_B(-C) \to \O_B \to \O_C \to 0$$
shows that $H^1(\O_B) = 0$. Now, the exact sequence
$$0 \to \O_B \to \G(\det\F) \to \I_{\{P\}/B}(C) \to 0$$
gives, using Lemma \ref{ulr}(vii), 
$$h^0(\I_{\{P\}/B})=h^0(\G(\det\F))-1=h^0(\E)-1=2d-1.$$ 
Then, the exact sequence  
$$0 \to \O_B \to \I_{\{P\}/B}(C) \to \O_{\P^1}(C^2-1) \to 0$$
shows, since $C^2 \ge 0$, that 
$$c_1(\G(\det\F))^2=C^2=h^0(\O_{\P^1}(C^2-1))=h^0(\I_{\{P\}/B})-1=2d-2.$$ 
On the other hand, since $\E$ is semistable by \cite[Thm.~2.9]{ch}, it satisfies Bogomolov's inequality 
$0 \le (4c_2(\E)-c_1(\E)^2)H^{n-2}$, that is
$$0 \le \pi^*(4c_2(\G(\det\F))-c_1(\G(\det\F))^2)H^{n-2}=4c_2(\G(\det\F))-c_1(\G(\det\F))^2=6-2d$$ 
that is $d \le 3$. Since $\rho(X) \ge 2$ we deduce that $d=3$ and $(X, \O_X(1)) \cong (\P^1\times\P^2, \O_{\P^1}(1) \boxtimes \O_{\P^2}(1))$ (see for example \cite[Thm.~3.1]{h2}). Also, $c_1(\E)^n=\pi^*c_1(\G(\det\F))^n=0$ and we find that $\E=q^*(\O_{\P^2}(1)^{\oplus 2})$ by \cite[Cor.~4.9]{ls}, where $q$ is the second projection. This concludes the proof.
\end{proof}

\section{Surfaces}

In the case of surfaces, Ulrich subvarieties are $0$-dimensional and smooth, hence we need to understand when they can be a point. 

Before giving a series of examples, we recall the following fact.

\begin{remark}
\label{cub1}
Let $\Gamma \subset \P^3$ be a smooth cubic surface. Then there are $72$ classes of twisted cubics on $\Gamma$, listed in \cite[Ex.~3.5]{ch}. Note that, for any two twisted cubics $T, T'$ with $T \not\sim T'$, we have that $T \cdot T' \ge 2$.
In fact, assume that $m:=T \cdot T' \le 1$. Since $h^0(\O_{\Gamma}(T))=3$, the exact sequence
$$0 \to \O_{\Gamma}(T'-T) \to  \O_{\Gamma}(T') \to \O_{\P^1}(m) \to 0$$
shows that $h^0(\O_{\Gamma}(T'-T)) \ge 3-h^0(\O_{\P^1}(m)) \ge 1$. Hence $T'-T$ is effective and $H \cdot (T'-T)=0$, implying that $T \sim T'$.
\end{remark}

\begin{example}
\label{cub}
Let $\Gamma \subset \P^3$ be a smooth cubic surface. All rank $2$ Ulrich bundles $\E$ on $\Gamma$ with $c_2(\E)=1$ are of type $\E=\O_{\Gamma}(T)^{\oplus 2}$, where $T \subset \Gamma$ is a twisted cubic (as listed in  \cite[Ex.~3.5]{ch}). 

Indeed, it follows from \cite[Ex.~3.6]{ch} that $\E$ is an extension of type
$$0 \to \O_{\Gamma}(T) \to \E \to \O_{\Gamma}(T') \to 0$$
where $T$ and $T'$ are two twisted cubic curves contained in $\Gamma$. Then $1=c_2(\E)=T \cdot T'$ and Remark \ref{cub1} implies that $T \sim T'$. On the other hand $\Ext^1(\O_{\Gamma}(T'),\O_{\Gamma}(T))=H^1(\O_{\Gamma}(T-T'))=H^1(\O_{\Gamma})=0$ and therefore the above sequence splits, hence $\E=\O_{\Gamma}(T)^{\oplus 2}$.
\end{example}

Next, we recall the notation for Hirzebruch surfaces $X_e = \P(\O_{\P^1} \oplus \O_{\P^1}(-e))$ with $e \ge 0$, $f$ a ruling and $C_0$ an irreducible curve with $C_0^2=-e, C_0 \cdot f=1$. In particular, a smooth non-degenerate cubic $S \subset \P^4$ is isomorphic to the Hirzebruch surface $X_1$ embedded with $H=C_0+2f$.

\begin{example}
\label{cusc}
Let $\Sigma \subset \P^4$ be a smooth non-degenerate cubic surface. Then $\E = \O_{\Sigma}(C_0+f)^{\oplus 2}$ is a rank $2$ Ulrich bundle with $c_2(\E)=1$. 

In fact, as said above, we have that $(\Sigma,H) \cong (X_1,C_0+2f)$. Now, $H^0(\O_{X_1}(C_0+f-H))=H^0(\O_{X_1}(-f))=0$ and, by Serre duality, $H^2(\O_{X_1}(C_0+f-2H))=H^0(\O_{X_1}(-C_0))^*=0$. Therefore $H^0(\E(-H))=H^2(\E(-2H))=0$. Moreover $c_1(\E)=2C_0+2f$ and $c_2(\E)=(C_0+f)^2=1$ satisfy \cite[(2.2)]{ca}. Hence $\E$ is Ulrich by \cite[Prop.~2.2(4)]{ca}. 
\end{example}

As it turns out, the above examples are the only ones with connected Ulrich subvarieties, as we will see in the proof of Theorem \ref{superficie}. Before proving the theorem, we need two lemmas.

\begin{lemma} 
\label{sup2-bis}
Let $S \subset \P^N$ be a smooth irreducible surface of degree $d \ge 2$ and let $\E$ be a rank $r \ge 2$ Ulrich bundle on $S$. Let $Z$ be a nonempty Ulrich subvariety associated to $\E$. Then $Z$ is connected if and only if $c_2(\E)=1$. In the latter case, we have that $r=2, c_1(\E)^2=2d-2$ and $2 \le d \le 3$.
\end{lemma}
\begin{proof}
Since $\dim Z =0$, $Z$ is smooth by definition and $[Z]=c_2(\E)$, we have that $Z$ is connected if and only if $Z=\{P\}$ is a point, that is, if and only if $c_2(\E)=1$. For the rest of the proof we assume that $c_2(\E)=1$. Let $V \subset H^0(\E)$ be a general subspace of dimension $r-1$ and let $\varphi : V \otimes \O_S \to \E$. It follows from Remark \ref{usv} that $Z=D_{r-2}(\varphi)$ is an Ulrich subvariety associated to $\E$. In particular, $Z \ne \emptyset$ by Proposition \ref{c2=0} and since $[Z]=c_2(\E)$ and $Z$ is smooth, we have that $Z=\{P\}$ for some point $P \in S$. Now Theorem \ref{connesse} implies that $r=2$ and Lemma \ref{zeta1} gives a smooth irreducible curve $C \in |\det \E|$ with $P \in C$ and $\O_C(P)$ globally generated, so that $C \cong \P^1$. Note that $C^2=c_1(\E)^2 > 0$, for otherwise we have the contradiction $c_2(\E)=0$ by Lemma \ref{c3no}(i). Therefore $H^1(\O_S(-C))=0$ by Kawamata-Viehweg's vanishing and the exact sequence
$$0 \to \O_S(-C) \to \O_S \to \O_C \to 0$$
implies that $H^1(\O_S)=0$. From the exact sequence
$$0 \to \O_S \to \E \to \I_{\{P\}/S}(C) \to 0$$
we deduce, using Lemma \ref{ulr}(vii), that $h^0(\I_{\{P\}/S}(C))=h^0(\E)-1=2d-1$. Therefore the exact sequence
$$0 \to \O_S \to \I_{\{P\}/S}(C) \to \O_{\P^1}(C^2-1) \to 0$$
gives that 
$$c_1(\E)^2=C^2=h^0(\O_{\P^1}(C^2-1))=h^0(\I_{\{P\}/S}(C))-1=2d-2.$$
On the other hand, since $\E$ is semistable by \cite[Thm.~2.9]{ch}, it satisfies Bogomolov's inequality $4c_2(\E)-c_1(\E)^2 \ge 0$, that is $0 < 2d-2=c_1(\E)^2 \le 4$, so that $2 \le d \le 3$.
\end{proof}

\begin{lemma} 
\label{hir1-bis}
Let $(S,H)=(X_1,C_0+2f)$ be a Hirzebruch surface. Then, the only  rank $2$ Ulrich bundle $\E$ with $c_2(\E)=1$ is $\E = \O_{X_1}(C_0+f)^{\oplus 2}$. 
\end{lemma}
\begin{proof}
We know by Example \ref{cusc} that $\E = \O_{X_1}(C_0+f)^{\oplus 2}$ is Ulrich on $(X_1,C_0+2f)$. Now, assume that $\E$ is a rank $2$ Ulrich bundle $\E$ with $c_2(\E)=1$ and let $\det \E = \O_X(D)$. Then $D=\alpha C_0+\beta f$ and it follows from Lemma \ref{sup2-bis} that $D^2=4$, hence $\alpha \ge 1$ since $D$ is nef by Lemma \ref{ulr}(iii) and $4 = D^2 = \alpha(2\beta-\alpha)$ implies that $\alpha=\beta=2$. It follows from \cite[Thm.~1.1(1)]{a} that $\E = \O_{X_1}(C_0+f)^{\oplus 2}$.
\end{proof}

We now prove Theorem \ref{superficie}.

\renewcommand{\proofname}{Proof of Theorem \ref{superficie}}
\begin{proof}
We know by Lemma \ref{sup2-bis} that $Z$ is connected if and only if $c_2(\E)=1$. In the cases (i)-(iii) we therefore have that $Z$ is connected by Lemma \ref{cn=1-tris}(i) and Examples \ref{cub}, \ref{cusc}. Vice versa, assume that $Z$ is connected, so that $c_2(\E)=1$ and $r=2, 2 \le d \le 3$ by Lemma \ref{sup2-bis}. If $d=2$ we conclude by Lemma \ref{cn=1-tris}(i) that we are in case (i). If $d=3$ we have that $N \le 4$. If $N=3$ we get by Example \ref{cub} that we are in case (ii), while if $N=4$ we know by Lemma \ref{hir1-bis} that we are in case (iii). 
\end{proof}
\renewcommand{\proofname}{Proof}

\section{Threefolds}

In the case $\E$ is an Ulrich bundle on a threefold, we have a lot of information about connectedness of Ulrich subvarieties.

First, we can characterize precisely the non-big case.
  
\begin{prop}
\label{non big}
Let $X \subset \P^N$ be a smooth irreducible threefold and let $\E$ be a non-big Ulrich bundle on $X$ with $c_2(\E) \ne 0$. Then:
\begin{itemize}
\item[(i)] If $c_1(\E)^3>0$, every Ulrich subvariety associated to $\E$ is connected.
\item[(ii)] If $c_1(\E)^3=0$, the Ulrich subvarieties associated to $\E$ are connected if and only if $(X, \O_X(1), \E) = (\P^1 \times \P^2, \O_{\P^1}(1) \boxtimes \O_{\P^2}(1), q^*(\O_{\P^2}(1)^{\oplus 2}))$, where $q : \P^1 \times \P^2 \to \P^2$ is the second projection.
\end{itemize}
\end{prop}
\begin{proof}
Let $\det \E = \O_X(D)$. If $c_1(\E)^3>0$, it follows from \cite[Thm.~2]{lm} that either $\E$ is as in Lemma \ref{exa} and we are done, or $(X,\O_X(1),\E)=(Q,\O_Q(1), \mathcal S)$ where $Q=Q_3$ and $\mathcal S$ is the spinor bundle. In the latter case, we have that $D=H$, hence $H^1(\E(-D))=H^1(\E(-1))=0$, hence any Ulrich subvariety is connected by Lemma \ref{h1=0}. This proves (i). In case (ii), we have by \cite[Thm.~2]{lm} and $c_2(\E) \ne 0$ that $(X, \O_X(1), \E)$ is a linear Ulrich triple over a surface and we conclude by Lemma \ref{lut2}.
\end{proof}

In the big case, even though our results are not conclusive, they strongly restrict the possibilities, as follows.

\begin{prop}
\label{big}
Let $X \subset \P^N$ be a smooth irreducible threefold, let $\E$ be a big rank $r \ge 2$ Ulrich bundle on $X$ and let $Z$ be an Ulrich subvariety associated to $\E$. Then $Z$ is nonempty and we have:
\begin{itemize}
\item[(i)] If $\E$ is V-big and $r \ge 3$, then $Z$ is connected.
\item[(ii)] If $\E$ is not V-big, then $(X, \O_X(1))$ is one of the following:
\begin{itemize}
\item[(a)] A linear $\P^2$-bundle over a smooth curve.
\item[(b)] A del Pezzo $3$-fold of degree $d$ with $3 \le d \le 7$.
\item[(c)] A quadric fibration over a smooth curve.
\item[(d)] A linear $\P^1$-bundle over a smooth surface.
\end{itemize}
\end{itemize}
Moreover, in case (b), then $Z$ is connected if $3 \le d \le 5$ and if $r=2$ and $6 \le d \le 7$.
\end{prop}
\begin{proof}
Ulrich subvarieties associated to $\E$ are nonempty by Remark \ref{usv3} and Lemma \ref{ulr}(viii). Also, $c_1(\E)^3 > 0$ by \cite[Rmk.~2.2]{lm}. First, (i) follows from Theorem \ref{connesse}(iv). If $\E$ is not V-big, that is $\B_+(\E)=X$, then $X$ is covered by lines by \cite[Thm.~2]{bu} and it follows from \cite[Thm.~1.4]{lp} (for (d) use also \cite[Thm.~0.2]{sv}) that $(X, \O_X(1))$ is as in (a)-(d) above. In case (b), consider the classification of del Pezzo threefolds (see for example \cite[pages 860-861]{lp}, \cite[Table, page 710]{f1}). If $3 \le d \le 5$, we have that $\Pic(X) \cong \Z H$, hence $Z$ is connected by Lemma \ref{h1=0}(ii). To do the case $r=2$ and $6 \le d \le 7$, we will use Lemma \ref{ulr}(vi). Suppose that $X \cong \P(T_{\P^2})$. Then $Z$ is connected by \cite[Main Thm.~for F]{cfm1} if $\E$ is indecomposable. If $\E$ is direct sum of two Ulrich line bundles, it is easy to show, using \cite[Cor.~2.7]{cfm1}, that the only possibility is when $c_1(\E)=2H$ and then $Z$ is connected by Lemma \ref{h1=0}(i). When $X$ is the blow-up of $\P^3$ at a point, it follows from \cite[Prop.~2.7 and Thm.~A]{cfm3} that $Z$ is connected. Finally, when $X \cong \P^1 \times \P^1 \times  \P^1$, denote by $H_i$ the pull-back of $\O_{\P^1}(1)$ by the $i$-th projection, $1 \le i \le 3$. We have that either $\E$ is indecomposable and $Z$ is connected by \cite[Thm.~B]{cfm2}, unless we are in \cite[Thm.~B, case (5)]{cfm2} or $\E$ is decomposable. Now, in case (5), $c_1(\E)H^2=10$, while Lemma \ref{ulr}(x), gives that $c_1(\E)H^2=12$, hence this case does not occur for Ulrich bundles. If $\E$ is decomposable, then it is direct sum of two Ulrich line bundles and we easily deduce, from \cite[Lemma 2.4]{cfm2}, that each direct summand must be of type $2H_i+H_j$, for $i \ne j$. It follows that $H^1(-2H_i-H_j)=0$ and hence $H^1(\E(-D)) \cong H^1(\E^*)=0$. Hence $Z$ is connected by Lemma \ref{h1=0}(i).
\end{proof}

We remark that examples as in (a) with $Z$ disconnected do exist, see Example \ref{ese}. We do not know if there are others.
 
\section{Some examples}
\label{dieci}
 
We give some explicit examples, with various vanishings of Chern classes.

The following is an example of a big rank $r \ge 2$ Ulrich bundle $\E$ with $c_2(\E) \ne 0, c_2(\E)^2=c_3(\E)=0$ and disconnected Ulrich subvarieties (even in rank $2$).
\begin{example}
\label{ese}
Let $B$ be a smooth irreducible curve of genus $g$, let $\F$ be a rank $n \ge 2$ very ample bundle on $B$, let $X=\P(\F)$ with bundle map $\pi: X \to B$ and let $H = \xi$ be the tautological line bundle on $X$. Note that $\xi^n \ge 2$, otherwise $X=\P^n$, a contradiction since $\rho(X)=2$. Let $M$ (respectively $\G$) be a line bundle (resp. a rank $r-1$ vector bundle, with $r \ge 2$) on $B$ such that $H^i(M)=H^i(\G)=0$ for $i \ge 0$. Let $\L = \xi+\pi^*M$. Then $\L$ is an Ulrich line bundle for $(X,\O_X(1))$: In fact $\L(-pH)=(1-p)\xi+\pi^*M$. Since $R^j \pi_*((1-p)\xi)=0$ for $j \ge 1, 1 \le p \le n$ and for $j=0, p \ge 2$, we get, for $i \ge 0$, that $H^i(\L(-pH))=0$ for $2 \le p \le n$ and $H^i(\L(-H))=H^i(M)=0$. Let $\E$ be any bundle sitting in an extension of type
$$0 \to \L \to \E \to \pi^*(\G(\det \F)) \to 0.$$
It follows from \cite[Lemma 4.1]{lo} that $\E$ is an Ulrich bundle for $(X,\O_X(1))$. Setting 
$$N=c_1(\G(\det \F))=c_1(\G)+(r-1)c_1(\F)$$ 
we have 
$$c_1(\E)=\L+\pi^*N, c_2(\E)=c_1(\L)c_1(\pi^*(\G(\det \F))=(\xi+\pi^*M)\pi^*N=\xi \pi^*N.$$
Since $\L$ and $\pi^*(\G(\det \F))$ are Ulrich, we have that $\L$ and $\pi^*N$ are globally generated. Moreover, $\pi^*N$ is not trivial by Lemma \ref{ulr}(iii). Therefore $c_2(\E)=\xi \pi^*N \ne 0$. Also, we have that $\xi^n+\xi^{n-1}\pi^*M>0$: In fact, if $g \ge 1$, we have that $\pi^*M$ is nef, hence $\xi^n+\xi^{n-1}\pi^*M \ge \xi^n>0$. If $g=0$, we have that $M=\O_{\P^1}(-1)$, hence $\xi^n+\xi^{n-1}\pi^*M=\xi^n-1>0$. Therefore
\begin{equation}
\label{c1n}
c_1(\E)^n=\L^n+\L^{n-1}\pi^*N=\xi^n+\xi^{n-1}\pi^*M+\L^{n-1}\pi^*N \ge \xi^n+\xi^{n-1}\pi^*M>0.
\end{equation} 
Moreover $c_2(\E)^2=\xi^2 \pi^*N^2=0$, and, for $i \ge 3$,
$$\begin{aligned}[t] 
c_i(\E) &= \sum\limits_{j=0}^i c_j(\L)c_{i-j}(\pi^*(\G(\det \F))=c_i(\pi^*(\G(\det \F))+c_{i-1}(\pi^*(\G(\det \F))c_1(\L)=\\
& = \pi^*c_i(\G(\det \F))+\pi^*c_{i-1}(\G(\det \F))c_1(\L)=0.
\end{aligned}$$
Now, assume that $n \ge 3$ and let $Z$ be any Ulrich subvariety associated to $\E$. Then $Z$ is nonempty by Proposition \ref{c2=0}. Let $D=\L+\pi^*N$ and consider the exact sequence
$$0 \to \O_X(-D)^{\oplus (r-1)} \to \E(-D) \to \I_{Z/X} \to 0.$$
Since $D$ is big by \eqref{c1n} and nef by Lemma \ref{ulr}(iii), we have that $H^1(\O_X(-D))=H^2(\O_X(-D))=0$ by Kawamata-Viehweg. Thus,
$$h^1(\I_{Z/X})=h^1(\E(-D))=h^1(\pi^*(\G(\det \F))(-D))+h^1(\L(-D)).$$
Now, $h^1(\pi^*(\G(\det \F))(-D))=h^1(\pi^*(\G(\det \F-N-M))(-\xi))=0$, since $R^j \pi_*(-\xi)=0$ for $j \ge 0$. Also,
$0=\chi(\G)=\deg \G+(r-1)(1-g)$, hence $\deg \G = (r-1)(g-1)$ and $\deg N = (r-1)(\deg \F + g-1)>0$, since $\F$ is very ample.
Therefore $h^0(-N)=0$ and $h^1(-N)=\deg N+g-1=(r-1)\deg \F + r(g-1)$. It follows that
$$h^1(\L(-D))=h^1(\pi^*(-N))=h^1(-N)=(r-1)\deg \F + r(g-1)$$
so that
$$h^1(\I_{Z/X})=(r-1)\deg \F + r(g-1).$$
Now the exact sequence
$$0 \to \I_{Z/X} \to \O_X \to \O_Z \to 0$$
and $h^1(\O_X)=h^1(\O_B)=g$ show that
$$h^0(\O_Z) \ge h^1(\I_{Z/X})+1-g=(r-1)(\deg \F +g-1) \ge 2$$
hence $Z$ is disconnected. Finally we show that $\E$ is big. In fact, using Lemma \ref{seg}, we have
$$s_n(\E^*)=[\xi+\pi^*(M+N)]^n-(n-1)[\xi+\pi^*(M+N)]^{n-2}\xi \pi^*N=\xi^n+\xi^{n-1}(\pi^*N+n\pi^*M)=$$
$$=\deg \F + \deg N + n \deg M=r\deg \F + \deg \G + n \deg M= r\deg \F + (r+n-1)(g-1)>0$$
both if $g \ge 1$ and if $g=0$ since in this case $\deg \F \ge n$. Thus $\E$ is big but not V-big if $r \ge 3$ by Theorem \ref{connesse}(iv). 

For an explicit instance of the above, let $X=\P^1 \times \P^{n-1} \subset \P^N$ 
embedded with the Segre embedding and let $\E = \O_X(n-1,0)^{\oplus (r-1)} \oplus \O_X(0,1)$.
\end{example}

The following is an example of a rank $r \in \{2, 3\}$ globally generated bundle $\E$ with $c_3(\E)=0, c_2(\E)^2 \ne 0, H^0(\E^*)=0$ and disconnected degeneracy locus $D_{r-2}(\varphi)$.

\begin{example}
\label{eses}
Let $Q \subset \P^4$ be a smooth quadric, let $W=\P^2 \times Q$ and consider the Segre embedding $W \subset \P^{14}$ given by $L =\O_{\P^2}(1) \boxtimes \O_Q(1)$. Note that $K_W = -3L$, so that $W$ is a Mukai $5$-fold that is scheme-theoretically cut out by quadrics in $\P^{14}$ (see for example \cite[Rmk.~26]{r}. Let $X \in |L|$ be a general hyperplane section of $W$ and let $H=L_{|X}$. It follows from \cite[Ex.~14.1.5]{bs} that $X$ contains exactly two fibers, that are a linear $\P^2$ in $\P^{14}$, say $F_i=\P^2 \times \{z_i\}, i=1, 2$, with $N_{F_i/X} \cong \Omega_{\P^2}(1)$. Note that $W \cong \P(\O_Q(1)^{\oplus 3})$ and $X=Z(\sigma)$, with $\sigma \in H^0(\O_Q(1)^{\oplus 3})$ a general section. In particular we can assume that $z_1, z_2$ are general points of $Q$. Let $Z = F_1 \sqcup F_2 \subset X$. We first observe that $\I_{Z/X}(H)$ is globally generated. In fact, we have a surjection $\I_{Z/W}(L) \to \I_{Z/X}(H)$, hence it is enough to prove that $\I_{Z/W}(L)$ is globally generated. On the other hand, we have that the linear span $\langle Z \rangle$ is a linear $\P^5$ in $\P^{14}$ and of course $\I_{\langle Z \rangle/\P^{14}}(1)$ is globally generated. Since we have a surjection $\I_{\langle Z \rangle/\P^{14}}(1) \to \I_{\langle Z \rangle \cap W/W}(L)$, we just need to prove that $\langle Z \rangle \cap W=Z$. To this end, let $w \in \langle Z \rangle \cap W$. Then $w \in \langle f_1, f_2 \rangle$, with $f_1, f_2 \in Z$. If $f_1, f_2 \in F_1$ (or in $F_2$), then $w \in \langle f_1, f_2 \rangle
\subset F_1 \subset Z$. If $f_1 \in F_1$ and $f_2 \in F_2$, then we observe that the line $\langle f_1, f_2 \rangle$ is not contained in $W$. In fact, it is well-known that the only lines in $W$ are of type $R \times \{p\}$, with $R \subset \P^2$ a line, or $\{p\} \times R$, with $R \subset Q$ a line. Both cases are excluded since $f_i=(x_i,z_i)$ with $x_1 \ne x_2, z_1 \ne z_2$. Therefore $\langle f_1, f_2 \rangle \not\subset W$ and since $W$ is scheme-theoretically cut out by quadrics, we can find a quadric $Q' \subset \P^{14}$ such that $W \subset Q'$ and $\langle f_1, f_2 \rangle \not\subset Q'$. But then $w \in \langle f_1, f_2 \rangle \cap W \subset \langle f_1, f_2 \rangle \cap Q' = \{f_1, f_2\} \subset Z$. Thus, we have proved that $\I_{Z/X}(H)$ is globally generated. Now we have
$$(\Lambda^2 N_{Z/X})(-H_{|Z}) = \O_Z(K_X-K_Z-H) \cong \O_Z(-3H-K_Z) \cong \O_Z$$
since, on each $F_i \cong \P^2$ we have that $\O_{F_i} \otimes \O_Z(-3H-K_Z) \cong \O_{\P^2}$. Moreover, $H^2(\O_X(-H))=0$ by Kodaira vanishing and therefore the Hartshorne-Serre correspondence gives a rank two vector bundle $\F$ on $X$, with $\det \F = H$, sitting in an exact sequence
\begin{equation}
\label{hs}
0 \to \O_X \to \F \to \I_{Z/X}(H) \to 0.
\end{equation}
Also, note that $H^1(\O_X)=0$ by the exact sequence
$$0 \to \O_W(-L) \to \O_W \to \O_X \to 0$$
and the facts that $H^1(\O_W)=0$ by K\"unneth, since $H^1(\O_{\P^2})=H^1(\O_Q)=0$ and $H^2(\O_W(-L))=0$ by Kodaira vanishing. Hence $\F$ is globally generated. Moreover $\F^* \cong \F(-H)$ and twisting \eqref{hs} we get
$$0 \to \O_X(-H) \to \F^* \to \I_{Z/X} \to 0$$
and Kodaira vanishing gives $H^0(\F^*)=0, h^1(\F^*)=h^1(\I_{Z/X})=h^0(\O_Z)-h^0(\O_X)=1$. Also, we have 
$c_2(\F)=[Z]=[F_1]+[F_2]$ and therefore $c_2(\F)^2=[F_1]^2+[F_2]^2=2$, since the self-intersection formula \cite[Appendix A, C.7]{h1} shows that $[F_i]^2=c_2(N_{F_i/X})=c_2(\Omega_{\P^2}(1))=1$. Finally, if $r=2$ we set $\E=\F$. If $r=3$, let $0 \ne e \in H^1(\F^*)$ and consider the associated extension
$$0 \to \O_X \to \E \to \F \to 0.$$
Since $\O_X$ and $\F$ are locally free, we have that so is $\E$: In fact $\SExt^i_{\O_X}(\E,\O_X)=0$ for $i >0$ because the same holds for $\O_X$ and $\F$. Moreover, $\E$ is globally generated since $H^1(\O_X)=0$ and
$c_3(\E)=c_3(\F)=0, c_2(\E)^2=c_2(\F)^2=2$. Moreover, dualizing the above exact sequence we get
$$0 \to \F^* \to \E^* \to \O_X \to 0$$
with cohomology sequence, using $H^0(\F^*)=0, H^1(\F^*)=\C e$,
$$0 \to H^0(\E^*) \to H^0(\O_X) \mapright{\delta} H^1(\F^*)$$
where, by construction, $\delta(1)=e$, so that $\delta$ is an isomorphism and $H^0(\E^*)=0$. The fact that $D_{r-2}(\varphi)$ is disconnected for any $\varphi : \O_X^{\oplus (r-1)} \to \E$ such that $D_{r-2}(\varphi)$ is reduced of pure codimension $2$, follows from Theorem \ref{connesse}. As a matter of fact, one of these degeneracy loci is just $Z$.
\end{example}

The following is an example of a rank $4$ Ulrich bundle $\E$ with $c_3(\E) \ne 0, c_2(\E)^2 \ne 0, c_4(\E)=0, H^0(\E^*)=0$ and disconnected degeneracy locus $D_{r-3}(\varphi)$.
\begin{example}
\label{eses2}
Let $Q \subset \P^5$ be a smooth quadric and let $\E= \mathcal S' \oplus \mathcal S''$. It is easily checked, using \cite[Rmk.~2.9]{o}, that $c_2(\E)=2(e_2+e_2')$, where $e_2, e_2'$ generate $H^4(Q,\Z)$. Also, $c_3(\E)=2H^3, c_4(\E)=0$ and $c_2(\E)^2=8$. Now, for any $\varphi : \O_X^{\oplus 3} \to \E$ such that $D_1(\varphi)$ is reduced of pure codimension $3$, we have that $D_1(\varphi)$ has exactly two connected components: it has at least two by Theorem \ref{connesse} and no more since $[D_1(\varphi)]=c_3(\E)=2H^3$. 
\end{example}

\section{Results in rank $2$}
\label{undici}

We can give a precise result when $H^1(\O_X)=0$. In particular, the next lemma shows that, in rank $2$, connectedness of degeneracy loci (or of Ulrich subvarieties), in the case $k=2$, is not governed by Chern classes as in the case of higher rank.

\begin{lemma}
\label{conn2}
Let $X$ be a smooth variety with $H^1(\O_X)=0$. Let $\E$ be a globally generated vector bundle of rank $2$ on $X$ with $c_2(\E) \ne 0$ and $H^0(\E^*)=0$ (which holds, in particular, if $\E$ is Ulrich). Let $h=h^1(\E^*)$ and let $\sigma \in H^0(\E)$ be a general section. Then we have:    
\begin{itemize}
\item[(i)] $Z(\sigma)$ has $h+1$ connected components. In particular $Z(\sigma)$ is connected if and only if $H^1(\E^*)=0$.
\item[(ii)] If $H^1(\E^*) \ne 0$, there is a globally generated bundle $\tilde \E$ of rank $h+2$ on $X$ with $H^0(\tilde\E^*)=H^1(\tilde\E^*)=0$, such that $\E$ arises as a quotient, with $Z(\sigma)=D_h(\psi)$,
$$0 \to \O_X^{\oplus h} \mapright{\psi} \tilde\E \to \E \to 0.$$
\end{itemize}
Vice versa, given a globally generated bundle $\tilde \E$ as in (ii), then $Z(\sigma)$ has $h+1$ connected components.
\end{lemma}
\begin{proof} Note that $Z(\sigma) \ne \emptyset$ by Lemma \ref{gg} and $h=h^1(\E^*)=h^1(\I_{Z(\sigma)/X})$ by Lemma \ref{numcc}(c). On the other hand, as $H^1(\O_X)=0$, we have the exact sequence 
$$0 \to H^0(\O_X) \to H^0(\O_{Z(\sigma)}) \to H^1(\I_{Z(\sigma)/X})\to 0$$
showing that $h^0(\O_{Z(\sigma)})=h+1$, that is (i). Now (ii) follows from Lemma \ref{numcc}(b). Vice versa, given a globally generated bundle $\tilde \E$ as in (ii), then $Z(\sigma)$ has $h+1$ connected components by Lemma \ref{numcc}(a).
\end{proof}

We also have the following.

\begin{lemma}
\label{conn3}
Let $\E$ be a rank $2$ globally generated bundle on $X \subset \P^N$ with $c_2(\E) \ne 0$ and $n \ge 3$. Let $\sigma \in H^0(\E)$ be a general section and set $Z=Z(\sigma)$. Then: 
\begin{itemize}
\item [(i)] If $c_1(\E)^3 \ne 0$ and $H^1(\E^*)=0$, then $Z$ is connected. 
\item [(ii)] Let $Y$ be as in Lemma \ref{zeta1} and assume that $Y$ is smooth (which holds, for example, if $n \le 3$ or if $c_2(\E)^2 = 0$ by Porteous's formula), $h^0(\O_Y(Z))=2$, $c_1(\E)^3 \ne 0, H^0(\E^*)=0$ and $H^1(\O_X)=0$. Then $Z$ is connected.
\item [(iii)] If $n \ge 4$, $\E$ is Ulrich and $c_1(\E)^3=0$, then $Z$ is disconnected. 
\item [(iv)] If $n \ge 4, c_3(\E)=0$, $\E$ is Ulrich and $X$ is subcanonical, then $Z$ is connected unless $(X, \O_X(1), \E)=(\P^2 \times \P^2, \O_{\P^2}(1) \boxtimes \O_{\P^2}(1), \pi^*(\O_{\P^2}(2))^{\oplus 2})$, where $\pi : \P^2 \times \P^2 \to \P^2$ is a projection, and, in the latter case, $Z$ has exactly $4$ connected components. 
\end{itemize}
\end{lemma}
\begin{proof}
If we set $\det \E = \O_X(D)$, then, since $\E$ has rank $2$, we have that $\E(-D) \cong \E^*$. Now (i) follows from Lemma \ref{h1=0}(i). Under hypothesis (ii), the exact sequence  \eqref{incl2} becomes
$$0 \to \E^* \to \O_X^{\oplus 2} \to \O_Y(Z) \to 0$$
and it follows that $H^1( \E^*)=0$, so that $Z$ is connected by (i). This proves (ii). (iii) is just Remark \ref{718}. Finally, (iv) follows from Remark \ref{numcc2}.
\end{proof}
Note that (iii) and (iv) above are not in conflict, even though if $c_1(\E)^3=0$, then $c_3(\E)=0$ by Lemma \ref{c3no}(i). In fact, as proved in Remark \ref{numcc2}, if $c_1(\E)^3=0$ and $X$ is subcanonical, then $(X, \O_X(1), \E)=(\P^2 \times \P^2, \O_{\P^2}(1) \boxtimes \O_{\P^2}(1), \pi^*(\O_{\P^2}(2))^{\oplus 2})$.

On the other hand, when $c_1(\E)^3 \ne 0$, we have that $Z$ can be both connected (for example on a linear Ulrich triple over a threefold) or disconnected (even when $c_1(\E)^n > 0$), as in Example \ref{ese}.

\end{document}